\documentclass[12pt]{amsart}

\newtheorem{theorem}{Theorem}[section]

\newtheorem{lemma}[theorem]{Lemma}

\newtheorem{proposition}[theorem]{Proposition}
\newtheorem{corollary}[theorem]{Corollary}

\theoremstyle{remark}

\DeclareMathOperator{\cont}{cont} 

\newcommand{\bC}{\mathbb{C}}
\newcommand{\bG}{\mathbb{G}}

\newcommand{\bN}{\mathbb{N}}

\newcommand{\bQ}{\mathbb{Q}}
\newcommand{\bR}{\mathbb{R}}
\newcommand{\bZ}{\mathbb{Z}}
\newcommand{\cA}{{\mathcal{A}}}

\newcommand{\cL}{{\mathcal{L}}}

\newcommand{\cO}{{\mathcal{O}}}

\newcommand{\tA}{\tilde{A}}

\newcommand{\tF}{\tilde{F}}
\newcommand{\tM}{\tilde{M}}
\newcommand{\tP}{\tilde{P}}
\newcommand{\tQ}{\tilde{Q}}

\newcommand{\ue}{\mathbf{e}}

\newcommand{\ux}{\mathbf{x}}
\newcommand{\uy}{\mathbf{y}}
\newcommand{\uz}{\mathbf{z}}

\newcommand{\Cmult}{\bC^\times}
\newcommand{\CT}{\bC[T]}
\newcommand{\disp}{\displaystyle}
\newcommand{\et}{\quad\mbox{and}\quad}
\newcommand{\Ga}{\bG_\mathrm{a}}
\newcommand{\Gm}{\bG_\mathrm{m}}

\newcommand{\LT}{L[T]}
\newcommand{\Qmult}{\bQ^\times}
\newcommand{\QT}{\bQ[T]}
\newcommand{\Qplus}{\bQ_+}
\newcommand{\Res}{\mathrm{Res}}
\newcommand{\uns}{\{1,\dots,s\}}
\newcommand{\Zs}{\bZ^s}
\newcommand{\ZT}{\bZ[T]}

\begin{document}

\baselineskip=16.5pt 

\title[Small value estimates]
{Small value estimates for the additive group}
\author{Damien ROY}
\address{
   D\'epartement de Math\'ematiques\\
   Universit\'e d'Ottawa\\
   585 King Edward\\
   Ottawa, Ontario K1N 6N5, Canada}
\email{droy@uottawa.ca}
\dedicatory{ \`A Michel Waldschmidt,\\
 infatigable voyageur et grand communicateur,\\
 avec mes meilleurs v{\oe}ux et ma plus grande estime,\\
 \`a l'occasion de son soixanti\`eme anniversaire.}
\thanks{Work partially supported by NSERC and CICMA}

\maketitle

\begin{abstract}
We generalize Gel'fond's criterion of algebraic independence to the
context of a sequence of polynomials whose first derivatives take
small values on large subsets of a fixed subgroup of $\bC$, instead
of just one point (one extension deals with a subgroup of $\Cmult$).
\end{abstract}

%
%

\section{Introduction}
\label{sec:intro}

A typical proof of algebraic independence starts with the
construction of a sequence of auxiliary polynomials taking small
values at many points of a finitely generated subgroup $\Gamma$ of a
commutative algebraic group $G$ defined over some algebraic
extension of $\bQ$. This data is analyzed by applying sequentially a
criterion of algebraic independence and a zero estimate.  The
criterion of algebraic independence first looks at each value
individually while the zero estimate is used to ensure that the
polynomials do not vanish on nearby points from a slight
perturbation $\tilde{\Gamma}$ of $\Gamma$. The outcome is a lower
bound for the transcendence degree over $\bQ$ of the field $K$
generated by the coordinates of the points of $\Gamma$.

For further progress it would be desirable to have a tool that
encompasses both the criterion and the zero estimate by looking at
these small values globally as values of polynomials on the group
$G$ instead of looking at them one at a time, as elements of the
field $K$.  In \cite{R2001}, we conjecture such a ``small value
estimate'' for the group $\Ga\times\Gm$, and prove that it is
equivalent to Schanuel's conjecture.  In \cite{R2002}, we further
extend these ideas to the group $\Ga\times E$ where $E$ is an
elliptic curve defined over $\bQ$.

The present paper mainly deals with small value estimates for the
additive group $\Ga$ as a first step towards these conjectures.  The
following theorem provides an overview of our main results.  In its
formulation, the symbols $i$ and $j$ are restricted to integers. We
also write $H(P)$ to denote the \emph{height} of a polynomial
$P\in\ZT$, and $P^{[j]}$ to denote its $j$-th divided derivative
(see \S\ref{sec:prelim} for the precise definitions).

\begin{theorem}
 \label{thm:main}
Let $\xi$ be a transcendental complex number, let $\beta$, $\sigma$,
$\tau$ and $\nu$ be non-negative real numbers, let $n_0$ be a
positive integer, and let $(P_n)_{n\ge n_0}$ be a sequence of
non-zero polynomials in $\ZT$ satisfying $\deg(P_n) \le n$ and
$H(P_n)\le \exp(n^\beta)$ for each $n\ge n_0$.  The following six
statements hold.
\begin{enumerate}
\item[1)]
Let $r$ be a non-zero rational number. Suppose that $\beta>1$,
$\sigma + \tau < 1$ and $\nu > 1 + \beta - \sigma - \tau$. Then for
infinitely many $n$, we have
\[
 \max\big\{ |P_n^{[j]}(\xi+ir)| \,;\,
            0\le i\le n^\sigma,\, 0\le j \le n^\tau \big\}
 > \exp(-n^\nu).
\]
\item[2)]
Let $r$ be a positive rational number with $r\neq 1$.  Suppose that
$\beta > 1 + \sigma$, $\sigma + \tau < 1$ and $\nu > 1 + \beta -
\sigma - \tau$.  Then for infinitely many $n$, we have
\[
 \max\big\{ |P_n^{[j]}(r^i\xi)| \,;\,
            0\le i\le n^\sigma,\, 0\le j \le n^\tau \big\}
 > \exp(-n^\nu).
\]
\item[3)]
Suppose that $\beta > 1$, $(3/4)\sigma + \tau < 1$ and $\nu > 1 +
\beta - (3/4)\sigma - \tau$.  Then for infinitely many $n$, we have
\[
 \max\big\{ |P_n^{[j]}(i\xi)| \,;\,
            0\le i\le n^\sigma,\, 0\le j \le n^\tau \big\}
 > \exp(-n^\nu).
\]
\item[4)]
Let $r$ be a non-zero rational number. Suppose that $\beta>1$,
$(4/3)\sigma + \tau < 1$ and $\nu > 1 + \beta - (4/3)\sigma - \tau$.
Then for infinitely many $n$, we have
\[
 \max\big\{ |P_n^{[j]}(i_1\xi+i_2r)| \,;\,
            0\le i_1,i_2\le n^\sigma,\, 0\le j \le n^\tau \big\}
 > \exp(-n^\nu).
\]
\item[5)]
Let $\eta\in\bC$ be algebraic over $\bQ(\xi)$ with
$\eta\notin\bQ\xi$. Suppose that $\beta > 1$, $(3/2)\sigma + \tau <
1$ and $\nu> 1 + \beta - \sigma - \tau$. Then for infinitely many
$n$, we have
\[
 \max\big\{ |P_n^{[j]}(i_1\xi+i_2\eta)| \,;\,
            0\le i_1,i_2\le n^\sigma,\, 0\le j \le n^\tau \big\}
 > \exp(-n^\nu).
\]
\item[6)]
Let $\eta\in\bC$. Suppose that $\beta> 1$, $\sigma < 1$ and $\nu > 3
+ \beta - (11/4)\sigma$. Then for infinitely many $n$, we have
\[
 \max\big\{ |P_n(i\xi+\eta)| \,;\,
            0\le i\le n^{\sigma} \big\}
 > \exp(-n^\nu).
\]
\end{enumerate}
\end{theorem}

The second statement of the theorem is the only small value estimate
that we shall prove for the multiplicative group $\Gm$.  All the
others concern the additive group $\Ga$.  The statements 1) and 2)
can be viewed as extensions of Theorem 2.3 of \cite{R2005} in a
context where the degree of the polynomials is unbounded and the
number of points of evaluation is small compared to the degree. When
$\sigma=0$, they essentially reduce to Proposition 1 of \cite{LR1}.
In view of Dirichlet box principle, both results show a best
possible dependence in the parameter $\nu$ (see Proposition
\ref{prop:Dirichlet} from Appendix \ref{sec:appendix1}). The
statement 3) is our main result. Dirichlet box principle shows that
it would be false for a value of $\nu$ smaller than
$1+\beta-\sigma-\tau$.  This shows a gap of $\sigma/4$ compared to
our actual lower bound on $\nu$.  Similarly 4), 5) and 6) show
respectively a gap of $(2/3)\sigma$, $\sigma$ and $2-(7/4)\sigma$ in
the dependence in $\nu$ compared to the box principle (see Appendix
\ref{sec:appendix1}).

The proof of all results proceeds by contradiction and ultimately
rely on a version of Gel'fond's criterion of algebraic independence
that we recall in the next section.   Section \ref{sec:result},
inspired from \cite[\S6]{LR1}, deals with estimates for the
resultant of polynomials in one variable taking into account the
absolute values of the first derivatives of these polynomials at the
points of a finite set $E$. Section \ref{sec:trans} borrow ideas
from the proof of zero estimates to give upper bound for the degree
and height of an irreducible polynomial dividing the first
derivatives of polynomials of the form $P(aT+b)$ where $P\in\QT$ is
fixed and $(a,b)$ runs through a finite subset of $\Qplus\times\bQ$.
These tools are combined in \S\ref{sec:basic} to prove 1), 2) and
4). Statement 5) is proved in \S\ref{sec:further} in a more general
form involving subgroups of arbitrary rank. Besides the tools that
have already been mentioned, its proof also uses the following
result established in \S\ref{sec:gcd} as a consequence of a
combinatorial result from \S\ref{sec:inter}:

\begin{theorem}
 \label{intro:thm:gcd}
Let $\beta$, $\delta$ and $\mu$ be positive real numbers with $\mu <
1 < \beta$, let $A$ be the set of all prime numbers $p$ with $p\le
n^\mu$, let $P$ be a non-zero polynomial of $\QT$ of degree at most
$n$ and height at most $\exp(n^\beta)$ with $P(0)\neq 0$, and let
$Q$ be the greatest common divisor of the polynomials $P(aT)$ with
$a\in A$. If $n$ is sufficiently large as a function of $\beta$,
$\delta$ and $\mu$, we have $\deg(Q)\le n^{1-\mu+\delta}$ and $H(Q)
\le \exp(n^{\beta-\mu+\delta})$.
\end{theorem}

The proof of 3) given in \S\ref{sec:multxi} furthermore uses the
following result proved as a consequence of another combinatorial
statement from \S\ref{sec:Zaran}, related to Zarankiewicz problem:

\begin{theorem}
 \label{intro:thm:prod}
Let $\alpha$, $\beta$, $\delta$ and $\mu$ be positive real numbers
with $2\mu < \alpha < \beta$.  For each integer $n\ge 1$, let $A_n$
denote the set of all prime numbers $p$ with $p\le n^\mu$, and $B_n$
denote the set of all prime numbers $p$ with $n^\mu< p\le n^{2\mu}$.
For infinitely many $n$, there exists no non-zero polynomial
$P\in\ZT$ of degree at most $n^\alpha$ and height at most
$\exp(n^\beta)$ satisfying $\prod_{a\in A_n}\prod_{b\in B_n}
|P(ab\xi)| \le \exp(-n^{\alpha+\beta+2\mu+\delta})$.
\end{theorem}

Finally, 6) is proved in \S\ref{sec:higher} as a consequence of 3)
after elimination of $\eta$ through a resultant, upon observing that
this resultant as well as its first derivatives are small at
multiples of $\xi$.

\medskip\noindent \textbf{Sketch of proof of 3):} In order to help
the reader find his way through this paper, we conclude this
introduction by a brief sketch of proof of 3). We proceed by
contradiction, assuming on the contrary that for each sufficiently
large $n$ the polynomial $P_n$ satisfies $|P^{[j]}(i\xi)| \le
\exp(-n^\nu)$ for $i=1,\dots,[n^\sigma]$ and $j=0,1,\dots,[n^\tau]$.
Without loss of generality, after division of each $P_n$ by a
suitable power of $T$, we may assume that these polynomials do not
vanish at $0$. Define $A_n$ and $B_n$ as in Theorem
\ref{intro:thm:prod} for the choice of $\mu=\sigma/4$, and let $Q_n$
be the greatest common divisor of the polynomials $P_n^{[j]}(aT)$
with $a\in A_n$ and $j=0,1,\dots,[n^\tau/2]$.  Upon observing that
the latter family of polynomials take small values at the points
$ab\xi$ with $a\in A_n$ and $b\in B_n$, along with their derivatives
of order at most $[n^\tau/2]$, we deduce that $\prod_{a\in A_n}
\prod_{b\in B_n} |Q_n(ab\xi)| \le \exp(-n^{1+\beta-\tau+5\delta})$
for some positive $\delta$ which is independent of $n$.  By Theorem
\ref{intro:thm:gcd}, we further know that $Q_n$ has degree at most
$n^{1-\sigma/4+\delta}$ and height at most
$\exp(n^{\beta-\sigma/4+\delta})$.  By a standard linearization
process described in \S\ref{sec:prelim}, we deduce that $Q_n$ admits
an irreducible factor $R_n$ satisfying $\prod_{a\in A_n} \prod_{b\in
B_n} |R_n(ab\xi)| \le
\exp\big(-n^{\sigma/4-\tau+3\delta}(n^\beta\deg(R_n)+n\log
H(R_n))\big)$.  By independent means, we also know that $R_n$ has
degree at most $n^{1-\sigma/4-\tau+\delta}$ and height at most
$\exp(n^{\beta-\sigma/4-\tau+\delta})$.  Then, we deduce that there
exists a power $S_n$ of $R_n$ whose degree and height satisfy the
same estimates, with moreover $\prod_{a\in A_n} \prod_{b\in B_n}
|S_n(ab\xi)| \le \exp\big(-n^{1+\beta-2\tau+3\delta}\big)$. This
contradicts Theorem \ref{intro:thm:prod}.

%
%
\section{Notation and preliminaries}
\label{sec:prelim}

We denote respectively by $\Qmult$ and $\Cmult$ the multiplicative
groups of $\bQ$ and $\bC$. We also write $\Qplus$ for the
multiplicative group of positive rational numbers, and $\bN^*$ for
the set of positive integers.  We denote by $|E|$ the cardinality of
a set $E$.  Given subsets $A$ and $B$ of $\bC$, we write $A+B$
(resp.\ $AB$) to denote the set of all sums $a+b$ (resp.\ products
$ab$) with $a\in A$ and $b\in B$. Throughout this paper, the symbols
$i,j,k$ are restricted to integers. For $P\in\CT$ and $j\ge 0$, we
denote by $P^{[j]}$ the quotient by $j!$ of the $j$-th derivative of
$P$.

We define the \emph{norm} $\|\ux\|$ of any point $\ux$ in $\bC^n$ to
be its maximum norm.  Similarly, we define the \emph{norm} $\|P\|$
of a polynomial $P\in\bC[T_1,\dots,T_m]$ to be the maximum of the
absolute values of its coefficients.  When $\ux$ is a non-zero
element of $\bQ^n$, we define its \emph{content} $\cont(\ux)$ to be
the unique positive rational number $r$ such that $r^{-1}\ux$ is a
\emph{primitive} point of $\bZ^n$, namely a point of $\bZ^n$ with
relatively prime coordinates.  We also define its \emph{height}
$H(\ux)$ to be the ratio $\|\ux\|/\cont(\ux)$.  By extension, we
define respectively the \emph{content} $\cont(P)$ and \emph{height}
$H(P)$ of a non-zero polynomial $P\in\bQ[T_1,\dots,T_m]$ to be the
content and height of its coefficient vector.  Accordingly, we have
$H(P)=\|P\|/\cont(P)$. This notion of height is \emph{projective} as
we have $H(a\ux) = H(\ux)$ and $H(aP) = H(P)$ for any $a\in\Qmult$.
For a single rational number $x$, we adopt a slightly different
convention, and define its \emph{height} $H(x)$ to be the
inhomogeneous height of $x$, namely the height of the point
$(1,x)\in\bQ^2$.  This gives $H(x)=\max(|p|,|q|)$ if $p/q$ is the
reduced form of $x$.

We will frequently use the well-known fact that for any
$P_1,\dots,P_s\in\CT$ with product $P=P_1\cdots P_s$, we have
\[
 e^{-\deg(P)}\|P\| \le \|P_1\| \cdots \|P_s\| \le e^{\deg(P)} \|P\|
\]
\cite[Ch.~III, \S4, Lemma 2]{G}.  As the content is a multiplicative
function on $\QT\setminus\{0\}$, it follows that, for non-zero
polynomials $P_1,\dots,P_s\in\QT$, the same inequalities hold with
the norm replaced by the height.  This means that the height is
essentially multiplicative. In the sequel, we will also require the
following lemma which formalizes a standard procedure of
``linearization'':

\begin{lemma}
 \label{lemma:linearization}
Let $c$, $n$, $\rho$ and $X$ be positive real numbers with $e^n \le
X$, and let $\xi_1, \dots, \xi_s$ be a finite sequence of complex
numbers, not necessarily distinct. Suppose that there exists a
non-zero polynomial $P\in\QT$ of degree at most $\rho n$ and height
at most $X^\rho$ satisfying
\begin{equation}
 \label{lemma:lin:eq1}
 \prod_{i=1}^s \frac{|P(\xi_i)|}{\cont(P)}
 <
 \big( X^{\deg(P)} H(P)^n \big)^{-c/\rho}
\end{equation}
or the stronger condition
\begin{equation}
 \label{lemma:lin:eq2}
 \prod_{i=1}^s \frac{|P(\xi_i)|}{\cont(P)}
 <
 X^{-2cn}.
\end{equation}
Then, a) there exists an irreducible factor $R$ of $P$ in $\QT$
satisfying
\begin{equation}
 \label{lemma:lin:eq3}
 \prod_{i=1}^s \frac{|R(\xi_i)|}{\cont(R)}
 <
 \big( X^{\deg(R)} H(R)^n \big)^{-c/(2\rho)},
\end{equation}
and b) there exists an integer $k\ge 1$ such that the polynomial
$Q=R^k$ satisfies
\begin{equation}
 \label{lemma:lin:eq4}
 \deg(Q)\le \rho n,
 \quad
 H(Q) \le X^{2\rho}
 \et
 \prod_{i=1}^s \frac{|Q(\xi_i)|}{\cont(Q)}
 <
 X^{-cn/4}.
\end{equation}
\end{lemma}

Usually the data takes the form \eqref{lemma:lin:eq2}.  In replacing
it by the weaker condition \eqref{lemma:lin:eq1}, one gains that the
right hand side becomes essentially a multiplicative function of
$P$.  Part b) of the lemma shows that not much is lost in the
process, regardless of the value of $\rho$. However, for given $c$,
$n$ and $X$, the conclusion of Part a) gets stronger for small
values of $\rho$.

\begin{proof}
Upon replacing $n$ by $n/\rho$, $X$ by $X^{1/\rho}$ and $c$ by
$\rho^2c$, we may assume without loss of generality that $\rho=1$.
We also note that the strict inequality in \eqref{lemma:lin:eq1}
implies that $P$ is a non-constant polynomial.

a) Factor $P$ as a product $P=R_1\cdots R_u$ of irreducible
polynomials of $\QT$.  Since $H(R_1)\cdots H(R_u) \le
e^{\deg(P)}H(P)$, the condition \eqref{lemma:lin:eq1} implies
\[
 \prod_{j=1}^u \prod_{i=1}^s \frac{|R_j(\xi_i)|}{\cont(R_j)}
 <
 \big( X^{\deg(P)} (e^{\deg(P)}H(P))^n \big)^{-c/2}
 \le
 \prod_{j=1}^u
   \big( X^{\deg(R_j)} H(R_j)^n \big)^{-c/2}.
\]
Therefore there is at least one index $j$ for which the polynomial
$R=R_j$ satisfies \eqref{lemma:lin:eq3}.

b) Since $R$ divides $P$, we have $\deg(R)\le n$ and $H(R)\le e^n X
\le X^2$.  Let $k\ge 1$ be the largest integer for which the
polynomial $Q=R^k$ satisfies $\deg(Q)\le n$ and $H(Q)\le X^2$.
Taking the $k$-th power on both sides of \eqref{lemma:lin:eq3}, we
obtain
\[
 \prod_{j=1}^u \frac{|Q(\xi_i)|}{\cont(Q)}
 <
 \big( X^{\deg(Q)} H(R)^{kn} \big)^{-c/2}.
\]
If $\deg(Q)\ge n/2$, the right hand side of this inequality is
bounded above by $X^{-cn/4}$, and the conditions of
\eqref{lemma:lin:eq4} are all satisfied. Assume now that $\deg(Q)\le
n/2$.  Then we have $\deg(R^{2k})\le n$, and the choice of $k$
implies $H(R^{2k})\ge X^2$.  Since $H(R^{2k}) \le e^n H(R)^{2k}$, we
deduce that $H(R)^k \ge X^{1/2}$, and we reach the same conclusion.
\end{proof}

We conclude this section by stating the version of Gel'fond's
criterion of algebraic independence on which all our results
ultimately rely.

\begin{lemma}
 \label{gelfond:curve}
Let $\alpha$, $\beta$ and $\delta$ be positive real numbers with
$\beta\ge \alpha$, and let $\xi_1,\dots,\xi_m$ be a finite sequence
of complex numbers which generate a field of transcendence degree
one over $\bQ$. For infinitely many integers $n$, there exists no
polynomial $P\in \bZ[T_1,\dots,T_m]$ of degree at most $n^\alpha$
and height at most $\exp(n^\beta)$ satisfying
\[
 0 < |P(\xi_1,\dots,\xi_m)|
   \le \exp( -n^{\alpha+\beta+\delta} ).
\]
\end{lemma}

This follows for example from \cite[Theorem 2.11]{Ph} or \cite[\S7,
Corollary 3]{LR2}.  Alternatively, a standard norm argument reduces
the proof of this result to the case $m=1$ which is a direct
consequence of \cite[Theorem 1]{Br1}. The fact that one can separate
the estimates for the degree and height of the polynomials is an
original observation of D.~W.~Brownawell and M.~Waldschmidt which
played a crucial role in their proof of Schneider's eighth problem
\cite{Br2,Wa1}.  Note that, in the case $m=1$, the condition
$0<|P(\xi_1)|$ can be simply replaced by $P\neq 0$ since $\xi_1$ is
assumed to be transcendental over $\bQ$.

%
%
\section{Estimates for the resultant}
\label{sec:result}

For any finite subset $E$ of $\bC$ with at least two points, we
define
\begin{equation}
 \label{def:deltaE}
 \delta_E = \min_{\xi'\neq \xi} |\xi'-\xi|
 \et
 \Delta_E = \prod_{\xi'\neq \xi} |\xi'-\xi|^{1/2}
\end{equation}
where both the minimum and the product are taken over all ordered
pairs $(\xi',\xi)$ of distinct elements of $E$.  When $E$ consists
of one point, we put $\delta_E=\Delta_E=1$.  In the sequel, we will
often use the crude estimate $\Delta_E \ge
\min(1,\delta_E)^{(1/2)|E|^2}$. The main result of this section is
the following.

\begin{proposition}
 \label{result:propFG}
Let $n,s,t\in\bN^*$ with $n\ge st$, let $E$ be a set of $s$ complex
numbers, let $F$ and $G$ be non-zero polynomials of $\QT$ of degree
at most $n$ and let $Q\in\QT$ be their greatest common divisor.  For
any pair of integers $f$ and $g$ with $\deg(F/Q)\le f \le n$ and
$\deg(G/Q)\le g \le n$, we have
\begin{equation}
 \label{propFG:eq1}
 H(Q)^{f+g} \prod_{\xi\in E} \left( \frac{|Q(\xi)|}{\|Q\|} \right)^t
 \le
 c_1 H(F)^g H(G)^f
       \prod_{\xi\in E}
       \max_{0\le j<t} \max \left\{ \frac{|F^{[j]}(\xi)|}{\|F\|},
       \frac{|G^{[j]}(\xi)|}{\|G\|} \right\}^t,
\end{equation}
with $c_1 = e^{7n^2} (2+c_E)^{4nst} \Delta_E^{-t^2}$, where $c_E =
\max_{\xi\in E} |\xi|$ and $\Delta_E$ is defined above.
\end{proposition}

When $s=1$, this is essentially Lemma 13 of \cite{LR1}.  In other
words, we can view the above proposition as an extension of the
latter result dealing with values of polynomials and their
derivatives at several points instead of one.  The proof is similar
in that it proceeds through estimations of the resultant of $F/Q$
and $G/Q$. It will require several intermediate lemmas. Before going
into this, we note the following corollary.

\begin{corollary}
 \label{result:corPP}
Let $n,s,t\in\bN^*$ with $n\ge st$, let $E$ be a set of $s$ complex
numbers, let $P_1,\dots,P_r\in\QT$ be a finite sequence of $r\ge 2$
non-zero polynomials of degree at most $n$, and let $Q\in\QT$ be
their greatest common divisor.  Then we have
\begin{equation}
 \label{corPP:eq1}
 \prod_{\xi\in E} \left( \frac{|Q(\xi)|}{\cont(Q)} \right)^t
 \le
 e^{3n^2} c_1 \Big( \max_{1\le i\le r} H(P_i) \Big)^{2n}
       \prod_{\xi\in E} \Bigg(
       \max_{\substack{1\le i\le r \\ 0\le j<t}}
       \frac{|P_i^{[j]}(\xi)|}{\cont(P_i)}
       \Bigg)^t,
\end{equation}
where $c_1$ is as in Proposition \ref{result:propFG}.
\end{corollary}

\begin{proof}
Without loss of generality, we may assume that $P_1,\dots,P_r$ and
$Q$ have content $1$, or equivalently that they are primitive
polynomials of $\ZT$.  We may also assume that $Q(\xi)\neq 0$ for
each $\xi\in E$, and that $Q$ is not the gcd of any proper subset of
$\{P_1,\dots,P_r\}$. The latter condition implies that $r\le n+1$.
According to Lemma 12 of \cite{LR1} there exist integers
$a_1,\dots,a_r$ with $0\le a_i \le n$ for $i=1,\dots,r$ such that
$Q$ is the gcd of $F:=P_1$ and $G:=\sum_{i=1}^r a_i P_h$. Assuming,
as we may, that $a_1=0$, we find
\begin{equation*}
 \max\{H(F),H(G)\}
  \le \max\Big\{\|P_1\|, n \sum_{i=2}^r \|P_i\| \Big\}
  \le n^2 \max_{1\le i\le r} H(P_i)
\end{equation*}
and similarly, for any $\xi\in E$ and any $j=0,\dots,t-1$,
\begin{equation*}
 \max \big\{ |F^{[j]}(\xi)|,
       |G^{[j]}(\xi)| \big\}
  \le n^2 \max_{1\le i\le r} |P_i^{[j]}(\xi)|.
\end{equation*}
Applying Proposition \ref{result:propFG} with $f=g=n$, we then find
\begin{align*}
 \prod_{\xi\in E} |Q(\xi)|^t
 &\le H(Q)^{2n}
      \prod_{\xi\in E} \left( \frac{|Q(\xi)|}{\|Q\|} \right)^t \\
 &\le c_1 \max\{H(F),H(G)\}^{2n}
       \prod_{\xi\in E}
       \max_{0\le j<t} \max \big\{ |F^{[j]}(\xi)|,
       |G^{[j]}(\xi)| \big\}^t \\
 &\le c_1 (n^2)^{2n+st} \big( \max_{1\le i\le r} H(P_i) \big)^{2n}
       \prod_{\xi\in E} \Big(
       \max_{\substack{1\le i\le r \\ 0\le j<t}} |P_i^{[j]}(\xi)|
       \Big)^t.
\end{align*}
The conclusion follows using $n^2\le e^n$ and $st\le n$.
\end{proof}

In order to prove our main Proposition \ref{result:propFG}, we start
by establishing a simple technical lemma.

\begin{lemma}
 \label{result:lemmaTF}
Let $n,t\in\bN^*$, let $z,\xi\in\bC$, and let $F\in\CT$ be a
non-zero polynomial with $\deg(F) \le n$. Then, for each integer
$\ell \ge 0$, the polynomial $\tF(T)=(T-z)^\ell F(T)$ satisfies
\begin{equation}
 \label{lemmaTF:eq1}
 \max_{0\le j <t} \frac{|\tF^{[j]}(\xi)|}{\|\tF\|}
   \le e^{\deg(\tF)} (2+|\xi|)^\ell
       \max_{0\le j <t} \frac{|F^{[j]}(\xi)|}{\|F\|}.
\end{equation}
When $z=0$, we can omit the factor $e^{\deg(\tF)}$ in the upper
bound.
\end{lemma}

\begin{proof}
For any $j\ge 0$, we have
$
 \tF^{[j]}(\xi)
   = \sum_{h=0}^{\min(j,\ell)}
     \binom{\ell}{h} (\xi-z)^{\ell-h} F^{[j-h]}(\xi)
$
and so,
\[
 \max_{0\le j <t} |\tF^{[j]}(\xi)|
   \le (1+|\xi-z|)^\ell \max_{0\le j <t} |F^{[j]}(\xi)|.
\]
This leads to the required upper bound \eqref{lemmaTF:eq1} since
$\|\tF\| \ge e^{-\deg(\tF)} \max\{1,|z|\}^\ell \|F\|$.  When $z=0$,
we simply have $\|\tF\|=\|F\|$ and we may omit the factor
$e^{\deg(\tF)}$.
\end{proof}

\begin{lemma}
 \label{result:lemmaDet}
Let $m,s,t\in\bN^*$ with $m\ge st$, let $L$ be any field, let
$Q\in\LT$, and let $\xi_1,\dots,\xi_s$ be $s$ distinct elements of
$L$. Denote by $\LT_{\le m-1}$ the vector space of polynomials of
$\LT$ of degree at most $m-1$, and let $\varphi$ and $\psi$ the
$L$-linear maps from $\LT_{\le m-1}$ to $L^m$ which send a
polynomial $P\in\LT_{\le m-1}$ to the points $\varphi(P)$ and
$\psi(P)$ of $L^m$ whose $k$-th coordinates are respectively given
by
\begin{align*}
  \varphi(P)_k
   &= \begin{cases}
      (QP)^{[j]}(\xi_i)
       &\text{if $k=i+js$ with $1\le i\le s$ and $0\le j<t$,} \\
      P^{[k-1]}(0)
       &\text{if $st<k\le m$,}
      \end{cases} \\
  \psi(P)_k
   &= P^{[k-1]}(0) \quad \text{for $k=1,\dots,m$.}
\end{align*}
Then, for any choice of polynomials $P_1,\dots,P_m\in\LT_{\le m-1}$,
we have
\begin{equation}
 \label{lemmaDet:eq1}
 \det(\varphi(P_1),\dots,\varphi(P_m))
   = \pm \Delta^{t^2} \Big( \prod_{i=1}^s Q(\xi_i) \Big)^t
     \det(\psi(P_1),\dots,\psi(P_m))
\end{equation}
where $\Delta=\prod_{1\le i<j\le s}(\xi_j-\xi_i)$ if $s\ge 2$, and
$\Delta=1$ if $s=1$.
\end{lemma}

\begin{proof}
It suffices to show that \eqref{lemmaDet:eq1} holds for at least one
choice of $L$-linearly independent polynomials $P_1,\dots,P_m$. Put
$E(T)=(T-\xi_1)\cdots(T-\xi_s)$ and, for each $k=1,\dots,m$, define
\[
 P_k =
  \begin{cases}
   E(T)^j(T-\xi_1)\cdots(T-\xi_{i-1})
    &\text{if $k=i+js$ with $1\le i\le s$ and $0\le j<t$,} \\
   T^{k-1}
    &\text{if $st<k\le m$.}
  \end{cases}
\]
Then, each $P_k$ is a monic polynomial of degree $k-1$ and so the
$m\times m$ matrix whose rows are $\psi(P_1),\dots,\psi(P_m)$ is
lower triangular with all its diagonal entries equal to $1$.  This
gives $\det(\psi(P_1),\dots,\psi(P_m))=1$.  We claim that the matrix
with rows $\varphi(P_1),\dots,\varphi(P_m)$ has a block
decomposition of the form $\begin{pmatrix} U &0\\ M &I
\end{pmatrix}$ where $U$ is an upper triangular $st\times st$ matrix
and $I$ denotes the identity matrix of size $(m-st)\times(m-st)$. To
prove this, we fix indices $k,k'$ with $1\le k,k'\le m$.  If
$k,k'>st$, we find $\varphi(P_k)_{k'}=0$ when $k'\neq k$ and
$\varphi(P_k)_{k'}=1$ when $k'=k$.  If $k'>st\ge k$, we also find
$\varphi(P_k)_{k'}=0$ since $P_k$ has degree $k-1$. Suppose now that
$k'<k\le st$.  We can write $k=i+js$ and $k'=i'+j's$ with $1\le
i,i'\le s$ and $0\le j,j'<t$. Since $k'<k$, either we have $j'=j$
and $i'<i$ or we have $j'<j$. In both cases, we find that
$(T-\xi_{i'})^{j'+1}$ divides $QP_k$ and so $\varphi(P_k)_{k'} = 0$.
We also note that
\[
 \varphi(P_k)_k
  = \lim_{x\to \xi_i}
    Q(x) \frac{E(x)^j}{(x-\xi_i)^j} (x-\xi_1) \dots (x-\xi_{i-1})
  = Q(\xi_i) E'(\xi_i)^j (\xi_i-\xi_1) \dots (\xi_i-\xi_{i-1}).
\]
This proves the claim and also provides the value of the diagonal
elements of the matrix $U$.  Consequently we have
$\det(\varphi(P_1),\dots,\varphi(P_m)) = \det(U)$ where
\[
 \det(U)
  = \prod_{j=0}^{t-1} \prod_{i=1}^{s}
    Q(\xi_i) E'(\xi_i)^j (\xi_i-\xi_1)\dots(\xi_i-\xi_{i-1})
  =\pm\, \Delta^{t^2} \prod_{i=1}^s Q(\xi_i)^t.
\]
Thus \eqref{lemmaDet:eq1} holds for the present choice of
$P_1,\dots,P_m$ and therefore it holds in general.
\end{proof}

\begin{lemma}
 \label{result:lemmaResAB}
Let $n,s,t\in\bN^*$, and let $E$ be a set of $s$ complex numbers.
Let $F,G\in\CT$ be non-zero polynomials of degree at most $n$, and
let $Q\in\CT$ be their greatest common divisor.  Put $A=F/Q$,
$B=G/Q$, $a=\deg(A)$, $b=\deg(B)$ and $m=a+b$.  Finally, assume that
$m\ge st$. Then we have
\begin{equation}
 \label{lemmaResAB:eq1}
 |\Res(A,B)| \prod_{\xi\in E} \left( \frac{|Q(\xi)|}{\|Q\|} \right)^t
   \le c_2 \|A\|^b \|B\|^a \prod_{\xi\in E} \max_{0\le j<t}
      \max \left\{
        \frac{|F^{[j]}(\xi)|}{\|F\|},
        \frac{|G^{[j]}(\xi)|}{\|G\|} \right\}^t
\end{equation}
with $c_2 = m! (e(2+c_E))^{nst} \Delta_E^{-t^2}$, where $c_E =
\max_{\xi\in E} |\xi|$.
\end{lemma}

\begin{proof}
Let $\xi_1,\dots,\xi_s$ denote the $s$ elements of $E$.   By
definition of the resultant, we have $\Res(A,B) = \det(\psi(P_1),
\dots, \psi(P_m))$ where $\psi$ is defined as in Lemma
\ref{result:lemmaDet} for the choice of $L=\bC$, and where
$P_1,\dots,P_m$ stand for the sequence of polynomials
\[
 A(T),TA(T),\dots,T^{b-1}A(T),B(T),TB(T),\dots,T^{a-1}B(T).
\]
Applying Lemma \ref{result:lemmaDet}, we deduce that
\[
 |\Res(A,B)| \prod_{i=1}^s |Q(\xi_i)|^t = \Delta_E^{-t^2} |\det M|
\]
where $M$ is the $m\times m$ matrix with rows $\varphi(P_1), \dots,
\varphi(P_m)$ for the map $\varphi \colon \CT_{\le m-1} \to \bC^m$
defined in the lemma.  Let $\tM$ be the matrix obtained from $M$ by
dividing each of its first $b$ rows by $\|A\|$ and each of its last
$a$ rows by $\|B\|$.  Then, except in its first $st$ columns, all
coefficients of $\tM$ have absolute value at most $1$.  This implies
\[
 |\det M|
   = \|A\|^b \|B\|^a |\det \tM|
   \le m! \|A\|^b \|B\|^a C_1\cdots C_{st},
\]
where, for each $k=1,\dots,st$, we denote by $C_k$ the maximum norm
of the $k$-th column of $\tM$.  Fix a choice of $k$ as above and
write it in the form $k=i+js$ with $1\le i\le s$ and $0\le j<t$.
Applying Lemma \ref{result:lemmaTF} with $z=0$ together with the
estimate $\|F\| \le e^n \|A\| \|Q\|$, we find that the absolute
values of the first $b$ elements in the $k$-th column of $\tM$ are
bounded above by:
\[
 \max_{0\le \ell < b} \frac{|(T^\ell F)^{[j]}(\xi_i)|}{\|A\|}
  \le e^n \|Q\| \max_{0\le \ell < b} \frac{|(T^\ell F)^{[j]}(\xi_i)|}{\|F\|}
  \le (e(2+c_E))^n \|Q\| \max_{0\le j< t} \frac{|F^{[j]}(\xi_i)|}{\|F\|}.
\]
Upon replacing $b$ by $a$, $A$ by $B$ and $F$ by $G$ in the above
inequalities, we also get an upper bound for the absolute values of
the last $a$ elements in the $k$-th column of $\tM$.  This gives
\[
 C_k \le (e(2+c_E))^n \|Q\| \max_{0\le j< t} \max\Big\{
          \frac{|F^{[j]}(\xi_i)|}{\|F\|},
          \frac{|G^{[j]}(\xi_i)|}{\|G\|} \Big\}.
\]
The conclusion follows.
\end{proof}

\begin{lemma}
 \label{result:lemmaExt}
Lemma \ref{result:lemmaResAB} still holds if the hypothesis $m\ge
st$ is replaced by $n\ge st$ and the constant $c_2$ in
\eqref{lemmaResAB:eq1} is replaced by $c_3 = (2n)! (e(2+c_E))^{4nst}
\Delta_E^{-t^2}$ with the same value for $c_E$.
\end{lemma}

\begin{proof}
Since $c_3\ge c_2$, we may assume without loss of generality that $m
< st \le n$. Cauchy's inequalities show that all coefficients of a
polynomial have their absolute value bounded above by the supremum
norm of the polynomial on the unit circle of the complex plane.
Applying this to the polynomial $B$, we deduce that there exists
$z\in\bC$ with $|z|=1$ such that $\|B\| \le |B(z)|$. Put
\[
 \tF(T) = (T-z)^{st-a-b} F(T)
 \et
 \tA(T) = \frac{\tF(T)}{Q(T)} = (T-z)^{st-a-b} A(T).
\]
As $B(z)\neq 0$, we still have $\gcd(\tF,G)=Q$.  Since $\deg(\tA B)
= st$ and since $\tF$ and $G$ both have degree at most $2n$, Lemma
\ref{result:lemmaResAB} gives
\[
 |\Res(\tA,B)| \prod_{\xi\in E} \left( \frac{|Q(\xi)|}{\|Q\|} \right)^t
 \le  c \|\tA\|^b \|B\|^{st-b}
      \prod_{\xi\in E} \max_{0\le j< t} \max \left\{
        \frac{|\tF^{[j]}(\xi)|}{\|\tF\|},
        \frac{|G^{[j]}(\xi)|}{\|G\|} \right\}^t
\]
where $c = (st)! (e(2+c_E))^{2nst} \Delta_E^{-t^2}$.  On the other
hand, using the fact that the resultant is multiplicative in each of
its arguments and that $\Res(T-z,B) = \pm B(z)$, we find
\[
 |\Res(\tA,B)|
   = |B(z)|^{st-a-b} |\Res(A,B)|
   \ge \|B\|^{st-a-b} |\Res(A,B)|.
\]
The conclusion follows by combining the above two inequalities
together with $\|\tA\| \le 2^{st} \|A\|$ and the estimate
\[
 \max_{0\le j< t} \frac{|\tF^{[j]}(\xi)|}{\|\tF\|}
   \le (e(2+c_E))^{2n-b} \max_{0\le j< t} \frac{|F^{[j]}(\xi)|}{\|F\|},
\]
valid for each $\xi\in E$, which follows from Lemma
\ref{result:lemmaTF} using $\deg(\tF) \le n+st-b\le 2n-b$.
\end{proof}

\begin{proof}[Proof of Proposition \ref{result:propFG}]
Since both sides of the inequality \eqref{propFG:eq1} stay invariant
under multiplication of $F$, $G$ and $Q$ by non-zero rational
numbers, we may assume without loss of generality that $F$, $G$ and
$Q$ are primitive polynomials of $\ZT$.  Put $A=F/Q$ and $B=G/Q$.
Then, $A$ and $B$ are relatively prime primitive polynomials of
$\ZT$.  In particular, their resultant $\Res(A,B)$ is a non-zero
integer, and so we have $|\Res(A,B)| \ge 1$. Since $\|A\|=H(A)$ and
$\|B\|=H(B)$ are also positive integers, we deduce from Lemma
\ref{result:lemmaExt} that
\[
 \prod_{\xi\in E} \left( \frac{|Q(\xi)|}{\|Q\|} \right)^t
   \le c_3 H(A)^f H(B)^g \prod_{\xi\in E} \max_{0\le j<t}
      \max \left\{
        \frac{|F^{[j]}(\xi)|}{\|F\|},
        \frac{|G^{[j]}(\xi)|}{\|G\|} \right\}^t
\]
for any pair of integers $f$ and $g$ satisfying the conditions of
the proposition.  We also note that $c_3 \le e^{5n^2} (2+c_E)^{4nst}
\Delta_E^{-t^2}$ since $(2n)! \le e^{n^2}$ and $st\le n$. The
conclusion then follows the fact that $H(A) \le e^n H(F)/H(Q)$ and
$H(B) \le e^n H(G)/H(Q)$ since $F$ and $G$ both have degree at most
$n$.
\end{proof}

%
%

\section{Estimates for translates of polynomials}
\label{sec:trans}

For each $a\in\Qplus$ and each $b\in\bQ$, we denote by
$\lambda_{a,b}$ the automorphism of $\QT$ which maps a polynomial
$P\in\QT$ to
\begin{equation*}
 (\lambda_{a,b}P)(T)=P(aT+b).
\end{equation*}
This provides an injective map from $\Qplus\times\bQ$ to the group
of automorphisms of $\QT$, whose image is a subgroup $\cL$ of that
group. We define the \emph{height} of an element $\lambda_{a,b}$ of
$\cL$ by
\[
 H(\lambda_{a,b}) = H(1,a,b).
\]
Since $\lambda_{a,b}^{-1}=\lambda_{1/a,-b/a}$ and $H(1,1/a,-b/a) =
H(a,1,-b) = H(1,a,b)$, we have $H(\lambda^{-1}) = H(\lambda)$ for
any $\lambda\in\cL$. A similar computation shows that
$H(\lambda\lambda') \le 2 H(\lambda) H(\lambda')$ for any $\lambda,
\lambda'\in\cL$. Finally, we let the group $\cL$ act on $\bC$ by
\[
 \lambda_{a,b}\cdot\xi = a\xi + b,
\]
so that $(\lambda P)(\xi) = P(\lambda\cdot\xi)$ for any
$\lambda\in\cL$ and any $\xi\in\bC$.

\begin{lemma}
 \label{trans:lemmaH1}
Let $\lambda\in\cL$, let $P$ be a non-zero polynomial of $\QT$, and
let $n$ be an upper bound for the degree of $P$.  Then, we have
$\deg(\lambda P) = \deg(P) \le n$,
\[
 (3H(\lambda))^{-n}
 \le \frac{H(\lambda P)}{H(P)} \le (3H(\lambda))^{n}
 \et
 H(\lambda)^{-n}
 \le \frac{\cont(\lambda P)}{\cont(P)} \le H(\lambda)^{n}.
\]
\end{lemma}

\begin{proof}
Without loss of generality, we may assume that $P$ is a primitive
polynomial of $\ZT$ of degree $n$. It is clear that $\deg \lambda P
= n$. Choose $a\in\Qplus$ and $b\in \bQ$ such that
$\lambda=\lambda_{a,b}$, and let $q$ be the least common denominator
of $a$ and $b$, so that $H(\lambda)=|q|\max\{1,|a|,|b|\}$.  Since
$q^n\lambda(P) = q^n P(aT+b)$ has integer coefficients, we find
\[
 H(\lambda P)
 \le |q|^n \|P(aT+b)\|
 \le |q|^n(1+|a|+|b|)^n \|P\|
 \le (3H(\lambda))^n H(P),
\]
and $\cont(\lambda P) \ge |q|^{-n} \ge H(\lambda)^{-n} =
H(\lambda)^{-n}\cont(P)$.  The remaining inequalities follow from
the above with $\lambda$ replaced by $\lambda^{-1}$ and $P$ replaced
by $\lambda P$, using $H(\lambda^{-1})=H(\lambda)$.
\end{proof}

\begin{lemma}
 \label{trans:lemmaH2}
Let $a\in\Qplus$, $b\in\bQ$, and let $R$ be an irreducible
polynomial of $\QT$.  Then, $\lambda_{a,b} R$ is also an irreducible
polynomial of $\QT$. Moreover, it is an associate of $R$ if and only
if either we have $a\neq 1$ and $R$ is a rational multiple of
$(a-1)T+b$, or we have $(a,b)=(1,0)$.
\end{lemma}

\begin{proof}
The first assertion follows from the fact that $\lambda_{a,b}$ is an
automorphism of $\QT$ and that any automorphism of an integral
domain maps units to units and irreducible elements to irreducible
elements.

Suppose that $\lambda_{a,b} R$ is an associate of $R$.  Then
$\lambda_{a,b}$ permutes the roots of $R$ in $\bC$, and so there is
an integer $k\ge 1$ for which $\lambda_{a,b}^k$ fixes all roots of
$R$.  This means that the roots of $R$ are also roots of the
polynomial $\lambda_{a,b}^k(T)-T$.  If $a=1$, this polynomial is the
constant $kb$ and so we must have $b=0$.  If $a \neq 1$, it is a
non-zero polynomial of $\QT$ of degree $1$ with leading coefficient
$a^k-1 \neq 0$ (since $a>0$). In this case, $R$ must also be a
polynomial of degree $1$ and $\lambda_{a,b}$ fixes its root.  This
root must therefore be $b/(1-a)$ and so $R$ is a rational multiple
of $(a-1)T+b$.  The converse is clear.
\end{proof}

\begin{lemma}
 \label{trans:lemmaH3}
Let $n,s,t\in\bN^*$ with $st\le n$, let $\cA$ be a finite subset of
$\cL$ of cardinality at least $s$, let $P$ be a non-zero polynomial
of $\QT$ of degree at most $n$, and let $R$ be an irreducible
polynomial of $\QT$. Suppose that $R$ divides $\lambda(P^{[i]})$ for
each $\lambda\in\cA$ and each $i=0,1,\dots,t-1$. Then, we have
\begin{equation}
 \label{lemmaH3:eq1}
 \deg(R) \le n/(st)
 \et
 H(R) \le \big(3\max_{\lambda\in\cA} H(\lambda)\big)^{2n/(st)}
          H(P)^{1/(st)}.
\end{equation}
\end{lemma}

\begin{proof}
By Lemma \ref{trans:lemmaH2}, the polynomials $\lambda^{-1} R$ with
$\lambda\in\cA$ are all irreducible.  Suppose first that no two of
them are associates.  Let $\lambda\in\cA$.  Since $R$ divides
$\lambda(P^{[i]})$ for $i=0,1,\dots,t-1$, we deduce that
$\lambda^{-1}R$ divides $P^{[i]}$ for the same values of $i$, and
therefore that $(\lambda^{-1}R)^t$ divides $P$.  This being true for
each $\lambda\in\cA$, we conclude that $\prod_{\lambda\in\cA}
(\lambda^{-1} R)^t$ divides $P$.  Since, by Lemma
\ref{trans:lemmaH1}, the polynomials $\lambda^{-1} R$ have the same
degree as $R$ and height at least $(3H(\lambda))^{-\deg(R)} H(R)$,
we first deduce that $\deg(R) \le n/(st)$ and then that
\[
 H(P) \ge e^{-n} \prod_{\lambda\in\cA} H(\lambda^{-1} R)^t
      \ge \big(3e\max_{\lambda\in\cA} H(\lambda)\big)^{-n} H(R)^{st},
\]
which is stronger than \eqref{lemmaH3:eq1}.

Suppose now that there exist two distinct elements $\lambda'$ and
$\lambda''$ of $\cA$ for which $\lambda' R$ and $\lambda'' R$ are
associates.  Then, $R$ is an associate of $\lambda R$ for the
composite $\lambda = (\lambda')^{-1}\lambda''$.  Since $\lambda$ is
not the identity, Lemma \ref{trans:lemmaH2} shows that $R$ has
degree $1$, and using the explicit description of $R$ given by this
lemma we find $H(R) \le 2H(\lambda) \le 4H(\lambda')H(\lambda'')$.
Then, the inequalities \eqref{lemmaH3:eq1} are again satisfied
because of the hypothesis $n\ge st$.
\end{proof}

%
%

\section{Basic small value estimates}
\label{sec:basic}

In the preceding section, we introduced a group $\cL$ of
automorphisms of $\QT$ and an action of it on $\bC$.  With this
notation, the following proposition constitutes the first step in
all our small value estimates.

\begin{proposition}
 \label{basic:propQ}
Let $n,t\in\bN^*$, let $\cA$ be a non-empty finite subset of $\cL$,
and let $E$ be a non-empty finite subset of $\bC$.  Suppose that
$|E|t \le n$. Moreover, let $P$ be a non-zero polynomial of $\ZT$ of
degree at most $n$, and let $Q$ denote the greatest common divisor
in $\QT$ of the polynomials $\lambda(P^{[i]})$ with $\lambda\in\cA$
and $0\le i <t$. Then, upon putting
\[
 \delta_P = \max\{|P^{[j]}(\lambda\cdot\xi)| \,;\,
          \lambda\in\cA,\, \xi\in E,\, 0\le j< 2t \},
\]
we have
\[
 \prod_{\xi\in E} \left( \frac{|Q(\xi)|}{\cont(Q)} \right)^t
 \le
 \big( e^4 c_\cA (2+c_E) \big)^{4n^2}
   \Delta_E^{-t^2} H(P)^{2n}  \delta_P^{|E|t},
\]
where $c_\cA = \max_{\lambda\in\cA} H(\lambda)$ and $c_E =
\max_{\xi\in E} |\xi|$.
\end{proposition}

\begin{proof}
We apply Corollary \ref{result:corPP} to the family of polynomials
$\lambda(P^{[i]})$ with $\lambda\in\cA$ and $0\le i <t$. Each of
them has degree at most $n$ and, using Lemma \ref{trans:lemmaH1}, we
find that their height and content satisfy
\begin{align*}
 H\big( \lambda(P^{[i]}) \big)
 &\le (3H(\lambda))^n H(P^{[i]})
  \le (6H(\lambda))^n H(P)
  \le (6c_\cA)^n H(P), \\
 \cont\big( \lambda(P^{[i]}) \big)
 &\ge H(\lambda)^{-n} \cont(P^{[i]})
  \ge H(\lambda)^{-n} \cont(P)
  \ge c_\cA^{-n},
\end{align*}
where the last step in the second estimate comes from the hypothesis
that $P$ has integer coefficients. Moreover, upon writing
$\lambda=\lambda_{a,b}$ with $a\in\Qplus$ and $b\in\bQ$, we find for
each $\xi\in E$ and $j=0,1,\dots,t-1$,
\[
 \big( \lambda(P^{[i]}) \big)^{[j]}(\xi)
  = \binom{i+j}{i} a^j P^{[i+j]}(a\xi+b)
  = \binom{i+j}{i} a^j P^{[i+j]}(\lambda\cdot\xi)
\]
Since $|a| \le H(\lambda) \le c_\cA$, we deduce that
\[
 \frac{\big| \big( \lambda(P^{[i]}) \big)^{[j]}(\xi) \big|}
      {\cont\big( \lambda(P^{[i]}) \big)}
 \le
 2^{2t}c_\cA^{t+n}\delta_P
 \le
 (2c_\cA)^{2n} \delta_P.
\]
According to Corollary \ref{result:corPP}, this implies that
\[
 \prod_{\xi\in E} \left( \frac{|Q(\xi)|}{\cont(Q)} \right)^t
 \le
 c \big( (6c_\cA)^n H(P) \big)^{2n} \big( (2c_A)^{2n} \delta_P \big)^{|E|t},
\]
where $c = e^{10n^2} (2+c_E)^{4n|E|t} \Delta_E^{-t^2}$. The
conclusion follows using $|E|t\le n$.
\end{proof}

The next proposition analyzes the outcome of the preceding result
through the linearization process of \S\ref{sec:prelim}, with the
help of the degree and height estimates of Lemma
\ref{trans:lemmaH3}.

\begin{proposition}
 \label{basic:propR}
Let $n,s,t\in\bN^*$ with $st\le n$, let $\cA$ be a finite subset of
$\cL$, and let $E$ be a finite subset of $\bC$.  Suppose that
$\min(|\cA|,|E|) \ge s$.  Moreover, let $X$ be a real number with
\begin{equation}
 \label{propR:eqprop1}
 X
 \ge
 \max\{ 3^n, c_\cA^n, (2+c_E)^n, \delta_E^{-s^2t^2/n} \},
\end{equation}
where $c_\cA$ and $c_E$ are as in Proposition \ref{basic:propQ}, and
assume that there exists a non-zero polynomial $P$ of $\ZT$ of
degree at most $n$ and height at most $X$ satisfying
\begin{equation}
 \label{propR:eqprop2}
 \max\{|P^{[j]}(\lambda\cdot\xi)| \,;\,
    \lambda\in\cA,\, \xi\in E,\, 0\le j< 2t \}
 \le
 X^{-\kappa n/(st)}
\end{equation}
for some real number $\kappa > 27$. Then, there exist a primary
polynomial $S\in\QT$ and a point $\xi\in E$ with
\begin{equation}
 \label{propR:eqprop3}
 \deg(S)\le \frac{5n}{st},
 \quad
 H(S) \le X^{10/(st)}
 \et
 \frac{|S(\xi)|}{\cont(S)}
 \le
 X^{-\kappa'n/(st)^2},
\end{equation}
where $\kappa'=(\kappa-27)/16$.
\end{proposition}

\begin{proof}
Upon replacing $E$ by a smaller subset if necessary, we may assume
without loss of generality that $|E|=s$. Let $Q$ be a greatest
common divisor in $\QT$ of the polynomials $\lambda(P^{[i]})$ with
$\lambda\in\cA$ and $0\le i <t$. Since $\Delta_E \ge
\min(1,\delta_E)^{s^2}$, the condition \eqref{propR:eqprop1} implies
in particular that $\Delta_E^{-t^2} \le \min(1,\delta_E)^{-s^2t^2}
\le X^n$, and so Proposition \ref{basic:propQ} gives
\begin{equation}
 \label{propR:eq1}
 \prod_{\xi\in E} \left( \frac{|Q(\xi)|}{\cont(Q)} \right)^t
 \le
 X^{(27 - \kappa)n}
 =
 X^{-16\kappa' n}.
\end{equation}
Choose any $\lambda\in\cA$. Since $Q$ divides $\lambda P$, we find
\begin{equation}
 \label{propR:eq2}
 \deg(Q) \le \deg(\lambda P) \le n
 \et
 H(Q) \le e^n H(\lambda P)
      \le (3ec_\cA)^n H(P)
      \le X^4,
\end{equation}
where the estimates for the degree and height of $\lambda P$ come
from Lemma \ref{trans:lemmaH1}.   Therefore Lemma
\ref{lemma:linearization} applies to the present situation with
$\rho=4$ and $c=8\kappa'$.  It produces an irreducible factor $R$ of
$Q$ in $\QT$ with
\begin{equation}
 \label{propR:eq3}
 \prod_{\xi\in E} \left( \frac{|R(\xi)|}{\cont(R)} \right)^t
 \le
 \big( X^{\deg(R)} H(R)^n \big)^{-\kappa'}.
\end{equation}
Since $|E| = s$, there also exists a point $\xi\in E$ for which
\begin{equation}
 \label{propR:eq4}
 \frac{|R(\xi)|}{\cont(R)}
 \le
 \big( X^{\deg(R)} H(R)^n \big)^{-\kappa'/(st)}.
\end{equation}
Moreover, according to Lemma \ref{trans:lemmaH3}, the polynomial $R$
satisfies
\[
 \deg(R) \le \frac{n}{st}
 \et
 H(R) \le (3c_\cA)^{2n/(st)}H(P)^{1/(st)} \le X^{5/(st)},
\]
like all irreducible factors of $Q$.  Applying Lemma
\ref{lemma:linearization} to $R$ with $\rho = 5/(st)$ and $c =
5\kappa'/(st)^2$, we deduce that some power $S$ of $R$ has the
required properties \eqref{propR:eqprop3}.
\end{proof}

It is not possible in general to improve significantly on the
estimates \eqref{propR:eq2}.  We will see however that this can be
done when the set $\cA$ contains a collection of automorphisms of
the form $\lambda_{a,0}$ with $a$ in a multiplicatively independent
subset of $\Qplus$ (see \S\ref{sec:gcd}). This then brings a
significant improvement on \eqref{propR:eq3} which automatically
carries to \eqref{propR:eq4}. The later step going from
\eqref{propR:eq3} to \eqref{propR:eq4} can also be improved in some
instances by noting that the values of $R$ on the set $E$ cannot be
uniformly small (see \S\ref{sec:multxi}).

The next result proves the statements 1) and 4) of Theorem
\ref{thm:main} by choosing of $\sigma_1=0$ and $\sigma_2=\sigma$ for
Part 1), and $\sigma_1=\sigma_2=\sigma$ for Part 4).

\begin{theorem}
 \label{thm:ixi+ir}
Let $\xi$ be a transcendental complex number, let $r$ be a non-zero
rational number, and let $\beta$, $\sigma_1$, $\sigma_2$, $\tau$,
$\nu$ be non-negative real numbers with
\[
 \beta > 1,
 \quad
 \sigma_1\le 3\sigma_2,
 \quad
 (1/3)\sigma_1+\sigma_2+\tau < 1
 \et
 \nu > 1+\beta-(1/3)\sigma_1-\sigma_2-\tau.
\]
Then, there are arbitrarily large values of $n$ for which there
exists no non-zero polynomial $P\in\ZT$ of degree at most $n$ and
height at most $\exp(n^\beta)$ with
\[
 \max\{|P^{[j]}(i_1\xi+i_2r)| \,;\,
   1\le i_1\le n^{\sigma_1},\,
   1\le i_2\le n^{\sigma_2},\,
   0\le j\le n^\tau \}
 \le
 \exp(-n^{\nu}).
\]
\end{theorem}

\begin{proof}
Suppose on the contrary that such a polynomial $P$ exists for each
sufficiently large integer $n$, and choose a real number $\delta
>0$ such that
\[
 \nu \ge 1+\beta-(1/3)\sigma_1-\sigma_2-\tau + 4\delta
 \et
 1 \ge (1/3)\sigma_1+\sigma_2+\tau+3\delta.
\]
For a fixed large integer $n$ and a corresponding polynomial $P$,
define
\begin{align*}
 \cA
 &= \big\{ \lambda_{i_1,i_2r}\,;\,
    1\le i_1 \le n^{\sigma_1/3},\,
    0\le i_2 \le (1/2)n^{\sigma_2} \big\},   \\
 E
 &= \big\{ i_1\xi+i_2r \,;\,
    1\le i_1 \le n^{2\sigma_1/3},\,
    0\le i_2 \le (1/2)n^{\sigma_2-\sigma_1/3} \big\},
\end{align*}
where the values of $i_1$ and $i_2$ are restricted to integers.  Put
also
\[
 s=\big[ (2 + n^{\sigma_2+\sigma_1/3})/3 \big],
 \quad
 t=\big[ (1 + n^\tau)/2 \big],
 \quad
 X=\exp(n^\beta),
 \et
 \kappa=n^{3\delta}.
\]

Assume first that $\delta_E \ge X^{-n/(st)^2}$. Then, if $n$ is
sufficiently large, all the conditions of Proposition
\ref{basic:propR} are satisfied because we have $1\le st\le n$,
$|\cA|\ge s$, $|E|\ge s$, $c_\cA \le H(r) n^{\sigma_2}$ and $c_E \le
(|\xi|+|r|) n^{2\sigma_2}$, while the hypothesis on $P$ implies
\begin{equation*}
 \max\{|P^{[j]}(\lambda\cdot x)| \,;\,
    \lambda\in\cA,\, x\in E,\, 0\le j< 2t \}
 \le
 \exp(-n^{\nu})
 \le
 X^{-\kappa n/(st)}.
\end{equation*}
In this case, Proposition \ref{basic:propR} provides us with a
non-zero polynomial $S\in\QT$ and a point $x\in E$ satisfying
\[
 \deg(S)\le \frac{5n}{st},
 \quad
 H(S) \le \exp\left( \frac{10n^\beta}{st} \right)
 \et
 \frac{|S(x)|}{\cont(S)}
 \le
 \exp\left( - \frac{2n^{1+\beta+2\delta}}{s^2t^2} \right).
\]
Write $x=i_1\xi+i_2r$ with $i_1,i_2\in \bZ$ and put $Q =
\lambda_{i_1,i_2r}S$ so that $Q(\xi)=S(x)$.  Then, $Q$ is a non-zero
polynomial of $\QT$ and, using the crude estimates $1\le i_1\le n$
and $0\le i_2 \le n$, Lemma \ref{trans:lemmaH1} gives
$\deg(Q)=\deg(S)$,
\[
 \frac{H(Q)}{H(S)}
   \le (3nH(r))^{\deg(S)}
 \et
 \frac{\cont(S)}{\cont(Q)}
   \le (nH(r))^{\deg(S)}.
\]
Since $\beta>1$ and $st \le n$, the last two quantities are bounded
above by $\exp(n^\beta/(st)) \le \exp(n^{1+\beta}/(st)^2)$ for $n$
sufficiently large, and so the polynomial $Q$ satisfies
\begin{equation}
 \label{thm:ixi+r:eq1}
 \deg(Q)\le \frac{5n}{st},
 \quad
 H(Q) \le \exp\left( \frac{11n^\beta}{st} \right)
 \et
 \frac{|Q(\xi)|}{\cont(Q)}
 \le
 \exp\left( - \frac{n^{1+\beta+2\delta}}{s^2 t^2} \right).
\end{equation}

If $\delta_E < X^{-n/(st)^2}$, there exist integers $i_1$ and $i_2$
not both $0$ with absolute value at most $n$ such that
$|i_1\xi+i_2r|\le \exp(-n^{1+\beta}/(st)^2)$.  In this case, we
define $Q(T)=(i_1T+i_2r)^{[n/(st)]}$.  Since $n/(st) \ge
n^{3\delta}$, this polynomial satisfies \eqref{thm:ixi+r:eq1} if $n$
is large enough. Thus, for each sufficiently large $n$, there exists
a non-zero polynomial $Q\in\QT$ satisfying \eqref{thm:ixi+r:eq1}.
Since $s$ and $t$ behave like polynomials in $n$, this contradicts
Gel'fond's Lemma \ref{gelfond:curve}.
\end{proof}

\begin{proof}[Proof of Theorem \ref{thm:main}, part 2)]
Suppose on the contrary that for each sufficiently large integer
$n$, the polynomial $P_n$ satisfies
\begin{equation}
 \label{proof2:eq1}
 \max\big\{ |P_n^{[j]}(r^i\xi)| \,;\,
            0\le i\le n^\sigma,\, 0\le j \le n^\tau \big\}
 \le \exp(-n^\nu),
\end{equation}
and choose a positive real number $\delta$ such that $\nu \ge
1+\beta-\sigma-\tau+4\delta$. We claim that for any sufficiently
large integer $n$, all the hypotheses of Proposition
\ref{basic:propR} are satisfied with
\begin{gather*}
 P=P_n, \quad
 s=[(1+n^{\sigma})/2], \quad
 t=[(1+n^\tau)/2], \quad
 X=\exp(n^\beta), \quad
 \kappa=n^{3\delta}, \\
 A=\{1,r,\dots,r^{s-1}\}, \quad
 \cA=\{\lambda_{a,0}\,;\, a\in A\}
 \et
 E=\{\xi,r\xi,\dots,r^{s-1}\xi\}.
\end{gather*}
First of all, we have $\lambda\cdot x \in \{r^i\xi\,;\, 0\le i \le
n^\sigma\}$ for each $\lambda\in \cA$ and each $x\in E$, and so the
main condition \eqref{propR:eqprop2} of this proposition follows
from \eqref{proof2:eq1} and $X^{\kappa n/(st)} \le \exp(n^\nu)$. We
also find $c_\cA \le H(r)^s$, $c_E \le |\xi|\max(1,|r|)^s$ and
$\delta_E \ge c_1^s$ where $c_1=\min\{1,|r|,|r-1|,|\xi|\}$.  Since
$st\le n$, this implies that $3^n$, $c_\cA^n$, $c_E^n$ and
$\delta_E^{-s^2t^2/n}$ are all bounded above by $c_2^{ns}$ for some
constant $c_2\ge 1$. As $\beta > 1+\sigma$, this shows that the
technical condition \eqref{propR:eqprop1} is also satisfied if $n$
is large enough. Then, Proposition \ref{basic:propR} provides us
with a non-zero polynomial $S\in\QT$ and an integer $i$ with $0\le
i\le s-1$ satisfying
\[
 \deg(S)\le \frac{5n}{st},
 \quad
 H(S) \le \exp\left( \frac{10n^\beta}{st} \right)
 \et
 \frac{|S(r^i\xi)|}{\cont(S)}
 \le
 \exp\left( - \frac{2n^{1+\beta+2\delta}}{s^2t^2} \right).
\]
Put $Q(T) = S(r^iT) = \lambda_{r^i,0}S$, so that $Q(\xi)=S(r^i
\xi)$. Then, $Q$ is a non-zero polynomial of $\QT$ and Lemma
\ref{trans:lemmaH1} gives $\deg(Q)=\deg(S)$,
\[
 \frac{H(Q)}{H(S)}
   \le (3H(r)^s)^{\deg(S)}
 \et
 \frac{\cont(S)}{\cont(Q)}
   \le H(r)^{s\deg(S)}.
\]
As $\beta>1+\sigma$ and $st\le n$, these quantities are both bounded
above by $\exp(n^\beta/(st)) \le \exp(n^{1+\beta}/(st)^2)$ if $n$ is
sufficiently large, and then $Q$ satisfies \eqref{thm:ixi+r:eq1}.
Again this contradicts Gel'fond's Lemma \ref{gelfond:curve}.
\end{proof}

%
%

\section{Estimates for an intersection}
\label{sec:inter}

Throughout this section, we fix a positive integer $s$ and we denote
by $(\ue_1,\dots,\ue_s)$ the canonical basis of $\Zs$.  For each
$\ux\in\Zs$ and each subset $E$ of $\Zs$, we define
\[
 \cO(\ux)=\{\ux+\ue_1,\dots,\ux+\ue_s\}
 \et
 \cO(E) = \bigcup_{i=1}^s (E+\ue_i) = \bigcup_{\ux\in E} \cO(\ux),
\]
so that for subsets $E$ and $F$ of $\Zs$, we have
\begin{equation}
 \label{inter:eq1}
 E \subseteq (F-\ue_1)\cap \cdots \cap (F-\ue_s)
 \quad\Longleftrightarrow\quad
 \cO(E)\subseteq F.
\end{equation}
We are interested here in the following type of result.

\begin{proposition}
 \label{inter:propC1}
Let $E$ and $F$ be finite subsets of $\Zs$ with $\cO(E)\subseteq F$.
Suppose that $|F| \le s^2/4$.  Then, we have $|F| \ge (s/2)|E|$.
Moreover, if $\ux_1,\dots,\ux_r$ denote the distinct elements of
$E$, there is a partition $F=F_1\amalg\cdots\amalg F_r\amalg
F_{r+1}$ of $F$ such that, for each $i=1,\dots,r$, we have
$F_i\subseteq \cO(\ux_i)$ and $|F_i| \ge s/2$.
\end{proposition}

Note that, for any given finite set $E$, the equivalent conditions
\eqref{inter:eq1} hold with $F=\cO(E)$, and then we have $|F| \le s
|E|$. Thus any general estimate of the form $|F| \ge (s/c) |E|$ with
a constant $c\ge 1$ is optimal up to the value of $c$. The first
assertion of the proposition shows that we can take $c=2$ when $|F|
\le s^2/4$. We will show that similar estimates hold in general when
the cardinality of $F$ is at most polynomial in $s$, with similar
partitions of $E$ and $F$ into subsets of relatively small
diameters.  In Appendix \ref{sec:appendix2}, we show that, for any
pair of subsets $E$ and $F$ satisfying the slightly stronger
condition $E \subseteq F\cap (F-\ue_1)\cap \cdots \cap (F-\ue_s)$,
we have $|E| \le (1/s) |F| \log |F|$, but the proof does not provide
corresponding partitions for $E$ and $F$.

\begin{proof}
We first observe that the differences $\ue_i-\ue_j$ with $1\le
i,j\le s$ and $i\neq j$ are all distinct and thus, for any pair of
distinct points $\ux$, $\uy$ of $\Zs$, the set $\cO(\ux)\cap
\cO(\uy)$ contains at most one element.

Define $F_i = \cO(\ux_i) \setminus (\cO(\ux_1) \cup\cdots\cup
\cO(\ux_{i-1}))$ for $i=1,\dots,r$, and let $F_{r+1}$ denote the
complement of $F_1 \cup\cdots\cup F_r$ in $F$.  Then, $F_1,
\dots,F_{r+1}$ form a partition of $F$ with $F_i\subseteq
\cO(\ux_i)$ for $i=1,\dots,r$.  By virtue of the preceding
observation, we also have
\[
 |F_i|
  \ge |\cO(\ux_i)| - \sum_{j=1}^{i-1} |\cO(\ux_i)\cap\cO(\ux_j)|
  \ge s-(i-1) = s-i+1
\]
for each $i\le r$.  In particular, this gives $|F_i| \ge s/2$ for
$i=1,\dots,\min(r,[s/2]+1)$, and so $|F| \ge (s/2) \min(r,[s/2]+1)$.
Since $|F|\le s^2/4$, we conclude that $r \le s/2$ and thus that
$|F_i| \ge s/2$ for $i=1,\dots,r$.
\end{proof}

The statement of our main proposition requires more notation.  Given
any point $\ux = (x_1,\dots,x_s) \in \Zs$, we write $\|\ux\|_1$ to
denote its $\ell_1$-norm $|x_1|+\cdots+|x_s|$.  We also denote by
$U$ the subgroup of $\Zs$ given by
\[
 U = \{ (x_1,\dots,x_s) \in \Zs \,;\, x_1+\cdots+x_s=0 \}.
\]
For each integer $k\ge 0$, we define
\[
 C_k = \{ \ux \in U \,;\, \|\ux\|_1 \le 2k \},
\]
and observe, for later use, that any point of $C_k$ has at most $k$
positive coordinates and at most $k$ negative coordinates.  For any
point $\ux\in\Zs$ and any integer $k\ge 0$, we also define
\[
 C_k(\ux) = \ux + C_k.
\]

\begin{proposition}
 \label{inter:propC2}
Let $E$ and $F$ be finite subsets of $\Zs$ with $\cO(E)\subseteq F$.
Suppose that
\begin{equation}
 \label{propC2:eq1}
 |F| \le \frac{1}{2^{\ell+1}(\ell+1)!}\binom{s}{\ell+2}
\end{equation}
for some integer $\ell$ with $0\le \ell\le s-2$.  Then, we have
\begin{equation}
 \label{propC2:eq2}
 |F| \ge \frac{s-\ell}{2(\ell+1)} |E|.
\end{equation}
More precisely, if $E$ is not empty, there exist an integer $r\ge
1$, a sequence of points $\ux_1,\dots,\ux_r$ of $E$, and partitions
$E = E_1 \amalg\cdots\amalg E_r$ and $F = F_1 \amalg\cdots\amalg F_r
\amalg F_{r+1}$ of $E$ and $F$ which, for $i=1,\dots,r$ satisfy
\[
 \text{a)}\ E_i \subseteq C_\ell(\ux_i),
 \qquad
 \text{b)}\ F_i \subseteq \cO(E_i),
 \qquad
 \text{c)}\ |F_i| \ge \frac{s-\ell}{2(\ell+1)} |E_i|.
\]
\end{proposition}

The proof of this result requires three lemmas.

\begin{lemma}
 \label{inter:lemmaC1}
Let $k\ge 0$ be an integer, let $C$ be a subset of $C_k$ and let
$D=\cO(C)$.  Then, we have $|D| \ge ((s-k)/(k+1)) |C|$.
\end{lemma}

Note that, since $\ue_1+C \subseteq D$, we also have $|D| \ge |C|$.
Therefore, the conclusion of the lemma is interesting only when
$k<s/2$.

\begin{proof}
For each point $(\ux,i)\in C\times\uns$, we have either
$\|\ux+\ue_i\|_1 = \|\ux\|_1-1$ or $\|\ux+\ue_i\|_1 = \|\ux\|_1+1$.
Denote by $N$ the set of points $(\ux,i)$ in $C\times\uns$ which
satisfy the first condition, and by $P$ the set of those which
satisfy the second condition.  Since $N$ and $P$ form a partition of
$C\times \uns$, we have
\begin{equation}
 \label{lemmaC1:eq1}
 |N|+|P| = s |C|.
\end{equation}

For any fixed $\ux\in C$, the integers $i\in\uns$ such that
$(\ux,i)\in N$ are those for which the $i$-th coordinate of $\ux$ is
negative. Since $C\subseteq C_k$, such a point $\ux$ has at most $k$
negative coordinates and therefore there are at most $k$ values of
$i$ for which $(\ux,i)\in N$.  As this holds for any $\ux\in C$, we
deduce that
\begin{equation}
 \label{lemmaC1:eq2}
 |N| \le k |C|.
\end{equation}

Consider now the surjective map $\varphi\colon C\times \uns \to D$
given by $\varphi(\ux,i)=\ux+\ue_i$.  For any $(\ux,i)\in P$, the
$i$-th coordinate of $\varphi(\ux,i)$ is positive.  Since
$D\subseteq \cO(C_k)$, any point $\uy\in D$ has at most $k+1$
positive coordinates, and therefore we get $|P \cap
\varphi^{-1}(\uy)| \le k+1$ for each $\uy\in D$.  The surjectivity
of $\varphi$ then implies
\begin{equation}
 \label{lemmaC1:eq3}
 |P| = \sum_{\uy\in D} |P \cap \varphi^{-1}(\uy)| \le (k+1) |D|.
\end{equation}
The combination of \eqref{lemmaC1:eq1}, \eqref{lemmaC1:eq2} and
\eqref{lemmaC1:eq3} gives $(k+1)|D| \ge |P| = s|C|-|N| \ge
(s-k)|C|$, as announced.
\end{proof}

For any integer $k\ge 0$, any point $\ux\in\Zs$ and any subset $E$
of $\Zs$, we define
\[
 C_k(\ux,E) = C_k(\ux) \cap E
 \et
 D_k(\ux,E) = \cO(C_k(\ux,E)).
\]
With this notation, if a set $F$ contains $\cO(E)$, then it contains
$D_k(\ux,E)$ for any $k\ge 0$ and any $\ux\in\Zs$.  We can now state
the next lemma.

\begin{lemma}
 \label{inter:lemmaC2}
Let $E$ be a finite subset of $\Zs$.  For any integer $k\ge 0$ and
any point $\ux\in\Zs$, we have
\begin{itemize}
 \item[(i)] $\disp |D_k(\ux,E)| \ge \frac{s-k}{k+1} |C_k(\ux,E)|$,
 \\
 \item[(ii)] $\disp |D_k(\ux,E) \cap \cO(E\setminus C_k(\ux,E))|
             \le (k+1) |C_{k+1}(\ux,E)|$.
\end{itemize}
\end{lemma}

\begin{proof}
Fix a choice of $k$ and $\ux$, and put $C = C_k(\ux,E)-\ux$ and $D =
D_k(\ux,E)-\ux$.  Then, $C$ and $D$ are subsets of $\Zs$ with the
same cardinality as $C_k(\ux,E)$ and $D_k(\ux,E)$ respectively.
Since they satisfy the hypotheses $C\subseteq C_k$ and $D=\cO(C)$ of
Lemma \ref{inter:lemmaC1}, the inequality (i) follows directly from
this lemma.

To prove (ii), it suffices to show that, for any $\uy\in E\setminus
C_k(\ux,E)$ such that $D_k(\ux,E)\cap \cO(\uy) \neq \emptyset$, we
have $\uy\in C_{k+1}(\ux,E)$ and $|D_k(\ux,E) \cap \cO(\uy)| \le
k+1$.  Fix such a point $\uy$, assuming that there exists at least
one.  Since $D_k(\ux,E)\cap \cO(\uy) \neq \emptyset$, there is an
integer $i\in\uns$ such that $\uy+\ue_i \in D_k(\ux,E)$.  For this
choice of $i$, there is also a point $\uz\in C_k(\ux,E)$ and an
integer $j\in\uns$ such that $\uy+\ue_i=\uz+\ue_j$.  Rewriting this
equality in the form
\begin{equation}
 \label{lemmaC2:eq1}
 \uy-\ux = (\uz-\ux) + (\ue_j-\ue_i),
\end{equation}
we deduce that $\|\uy-\ux\|_1 \le \|\uz-\ux\|_1 + 2 \le 2(k+1)$ and
also that $\uy-\ux\in U$ since $U$ contains both $\uz-\ux$ and
$\ue_j-\ue_i$.  Since $\uy\in E$, this shows that $\uy\in
C_{k+1}(\ux,E)$.  Moreover, since $\uy\notin C_k(\ux,E)$, we also
have $\|\uy-\ux\|_1 > 2k$ and so $\|\uy-\ux\|_1 = 2k+2$, because
$\|\uy-\ux\|_1$ is an even integer.  This observation combined with
\eqref{lemmaC2:eq1} and the fact that $\|\uz-\ux\|_1 \le 2k$ tells
us that the $i$-th coordinate of $\uy-\ux$ is negative.  As
$\uy-\ux$ admits at most $k+1$ negative coordinates, we deduce that
there are at most $k+1$ values of $i$ such that $\uy+\ue_i \in
D_k(\ux,E)$, and so $|D_k(\ux,E) \cap \cO(\uy)| \le k+1$.
\end{proof}

\begin{lemma}
 \label{inter:lemmaC3}
Let $E$, $F$ and $\ell$ be as in the statement of Proposition
\ref{inter:propC2}. For each $\ux\in E$, there exists at least one
integer $k$ with $0\le k \le \ell$ such that
\begin{equation}
 \label{lemmeC3:eq1}
 |D_k(\ux,E) \cap \cO(E\setminus C_k(\ux,E))|
 \le
 \frac{s-k}{2(k+1)} |C_k(\ux,E)|.
\end{equation}
\end{lemma}

\begin{proof}
Suppose on the contrary that there exists $\ux\in E$ such that
\eqref{lemmeC3:eq1} does not hold for any $k$ with $0\le k\le \ell$.
Using Part (ii) of Lemma \ref{inter:lemmaC2}, this gives
\begin{equation*}
 2 (k+1) |C_{k+1}(\ux,E)|
   \ge 2 |D_k(\ux,E) \cap \cO(E\setminus C_k(\ux,E))|
   > \frac{s-k}{k+1} | C_k(\ux,E)|,
\end{equation*}
for $k=0,\dots,\ell$.  Multiplying these inequalities term by term
for all these values of $k$ and noting that $C_0(\ux,E)$ is the
singleton $\{\ux\}$, we deduce that
\[
 2^{\ell+1}(\ell+1)! |C_{\ell+1}(\ux,E)| > \binom{s}{\ell+1}.
\]
Since $F$ contains $D_{\ell+1}(\ux,E)$, the above estimate combined
with Lemma \ref{inter:lemmaC2} (i) leads to
\[
 |F| \ge |D_{\ell+1}(\ux,E)|
     \ge \frac{s-\ell-1}{\ell+2} |C_{\ell+1}(\ux,E)|
       > \frac{1}{2^{\ell+1}(\ell+1)!} \binom{s}{\ell+2},
\]
against the hypothesis \eqref{propC2:eq1} of Proposition
\ref{inter:propC2}.
\end{proof}

\begin{proof}[Proof of Proposition \ref{inter:propC2}]
Since the inequality \eqref{propC2:eq2} from the first assertion of
the proposition follows from the estimates c) of the second
assertion, it suffices to prove the latter.  To do so we proceed by
induction on $|E|$.  Let $\ux_1\in E$.  Lemma \ref{inter:lemmaC3}
combined with Part (i) of Lemma \ref{inter:lemmaC2} shows that there
exists an integer $k$ with $0\le k\le \ell$ such that, upon putting
$E_1 = C_k(\ux_1,E)$ and $F_1 = D_k(\ux_1,E) \setminus
\cO(E\setminus E_1)$, we have
\begin{align*}
 |F_1|
  &= |D_k(\ux_1,E)| - |D_k(\ux_1,E) \cap \cO(E\setminus E_1)| \\
  &\ge \frac{s-k}{k+1} |C_k(\ux_1,E)| - \frac{s-k}{2(k+1)}
       |C_k(\ux_1,E)|
   = \frac{s-k}{2(k+1)} |E_1| \ge \frac{s-\ell}{2(\ell+1)} |E_1|.
\end{align*}
Therefore the sets $E_1$ and $F_1$ fulfil the conditions a), b) and
c) of Proposition \ref{inter:propC2} for $i=1$.

If $E=E_1$, this proves the proposition with $r=1$ and
$F_2=F\setminus F_1$. In particular, the proposition is verified
when $|E|=1$.  Assume therefore that $E\neq E_1$.  We put $E' =
E\setminus E_1$ and $F' = F\setminus F_1$.  By construction, $F_1$
and $\cO(E')$ are disjoint sets.  Since $\cO(E')\subseteq \cO(E)
\subseteq F$, this implies that $\cO(E')\subseteq F'$.  Thus the
hypotheses of Proposition \ref{inter:propC2} are also satisfied by
$E'$ and $F'$ instead of $E$ and $F$, with the same value of $\ell$.
Since $|E'| < |E|$, we may assume by induction that there exists an
integer $r\ge 2$, a sequence of points $\ux_2,\dots,\ux_r$ of $E'$
and partitions $E' = E_2 \amalg\cdots\amalg E_r$ and $F' = F_2
\amalg\cdots\amalg F_{r+1}$ which fulfil the conditions a), b), c)
of the proposition for $i=2,\dots,r$.  Then the partitions $E = E_1
\amalg\cdots\amalg E_r$ and $F = F_1 \amalg\cdots\amalg F_{r+1}$
have all the required properties.
\end{proof}

%
%

\section{Estimates for the gcd}
\label{sec:gcd}

We say that a finite subset $A$ of $\Qplus$ with $s$ elements is
multiplicatively independent if it generates a free subgroup of
$\Qplus$ of rank $s$.  This happens for example when $A$ consists of
$s$ prime numbers. The main result of this section is the following
statement which immediately implies Theorem \ref{intro:thm:gcd}.

\begin{theorem}
 \label{gcd:thmG}
Let $A$ be a finite multiplicatively independent subset of $\Qplus$,
let $s$ be its cardinality, let $P$ be a polynomial of $\QT$ with
$P(0)\neq 0$, and let $Q = \gcd\{P(aT)\,;\, a\in A\}$. Suppose that
the number of non-associate irreducible factors of $P$ is at most
\[
 N(s,\ell) := \binom{s}{\ell+2} \frac{1}{2^{\ell+1} (\ell+1)!}
\]
for some integer $\ell$ with $0\le \ell \le s-2$. Then, we have
\begin{equation}
 \label{thmG:eq1}
 \deg(Q) \le \frac{2(\ell+1)}{s-\ell} \deg(P)
 \et
 \log H(Q) \le \frac{2(\ell+1)}{s-\ell} \big( \log H(P) + c_1 \deg(P) \big)
\end{equation}
where $c_1 = 8 + (4\ell+1)\log(c_A)$ and $c_A = \max_{a\in A} H(a)$.
\end{theorem}

The proof of the theorem proceeds first by a reduction to a specific
type of polynomial $P$.  To state and prove the lemma that we apply
for this purpose, we use the following notation.

For each $a\in\Qplus$, we simply write $\lambda_a$ to denote the
automorphism $\lambda_{a,0}$ of $\QT$ which maps a polynomial
$P\in\QT$ to $(\lambda_a P)(T)= P(aT)$ (see \S\ref{sec:trans}).
Moreover, for a given subgroup $G$ of $\Qplus$, we say that two
polynomials $P_1$ and $P_2$ of $\QT$ are \emph{$G$-equivalent} and
write $P_1\sim_G P_1$ if there exists $a\in G$ such that
$P_2=\lambda_a(P_1)$. We also say that a polynomial $P\in\QT$ is
$G$-pure if it can be written as a product of $G$-equivalent
irreducible polynomials of $\QT$.

\begin{lemma}
 \label{gcd:lemmaG}
Let $G$ be a subgroup of $\Qplus$, let $A$ be a finite subset of
$G$, let $P$ be a non-zero polynomial of $\QT$, and let $Q =
\gcd\{P(aT)\,;\, a\in A\}$.  Then, we can write $P$ as a product
$P=P_1\cdots P_N$ of $G$-pure polynomials $P_1,\dots,P_N$ with
simple roots so that $Q=\prod_{i=1}^N \gcd\{P_i(aT)\,;\, a\in A\}$.
\end{lemma}

\begin{proof} We first observe that $P$ can be written as a product
$P = P_1 \cdots P_M$ of polynomials $P_1,\dots,P_M$ with simple
roots such that $P_{i+1}$ divides $P_i$ for $i=1,\dots,M-1$, and
that such a factorization is unique up to multiplication of each
$P_i$ by an element of $\bQ^*$.  For each $a\in A$, the equality
$\lambda_a P = (\lambda_a P_1) \cdots (\lambda_a P_M)$ provides a
factorization of $\lambda_a P$ of the same type.  From this we
deduce that $Q=\prod_{i=1}^M \gcd\{P_i(aT)\,;\, a\in A\}$ is the
corresponding factorization of $Q$.  This reduces the proof of Lemma
\ref{gcd:lemmaG} to the case where $P$ has no multiple roots.

Let $R_1,\dots, R_L$ be a set of representatives for the equivalence
classes of $G$-equivalent irreducible factors of $P$.  We can also
write $P$ as a product $P = P_1 \cdots P_L$ of $G$-pure polynomials
$P_1,\dots,P_L$ such that for each $i=1,\dots,L$, all irreducible
factors of $P_i$ are $G$-equivalent to $R_i$.  Again, such a
factorization is unique up to multiplication of each $P_i$ by an
element of $\bQ^*$.  Moreover, for each $a\in A$, the corresponding
factorization of $\lambda_a P$ is $\lambda_a P = (\lambda_a P_1)
\cdots (\lambda_a P_L)$, and so we deduce that $Q=\prod_{i=1}^L
\gcd\{P_i(aT)\,;\, a\in A\}$.  This further reduces the proof of
Lemma \ref{gcd:lemmaG} to the case where $P$ is $G$-pure and so
completes the proof of the lemma.
\end{proof}

\begin{proof}[Proof of Theorem \ref{gcd:thmG}]
Let $G$ denote the subgroup of $\Qplus$ generated by $A$.  We claim
that the conclusion \eqref{thmG:eq1} of the theorem holds with the
constant $c_2 = c_1 - 2$ instead of $c_1$ when $P$ is $G$-pure with
no multiple factors.  If we take this for granted and apply it to
each factor in the factorization $P = P_1\cdots P_N$ of $P$ provided
by Lemma \ref{gcd:lemmaG}, we find that for each $i$ the polynomial
$Q_i=\gcd\{P_i(aT)\,;\, a\in A\}$ satisfies
\[
 \deg(Q_i) \le \rho \deg(P_i)
 \et
 \log H(Q_i) \le \rho \big( c_2 \deg(P_i) + \log H(P_i) \big)
\]
where $\rho=2(\ell+1)/(s-\ell)$.  Since Lemma \ref{gcd:lemmaG} gives
$Q = Q_1\cdots Q_N$, these inequalities in turn imply that $\deg(Q)
\le \rho \deg(P)$ and that
\begin{align*}
 \log H(Q)
 &\le \deg(Q) + \sum_{i=1}^N \log H(Q_i) \\
 &\le \rho \sum_{i=1}^N \big( (c_2+1)\deg(P_i) + \log H(P_i) \big)
 \\
 &\le \rho \big( (c_2+2)\deg(P) + \log H(P) \big),
\end{align*}
as announced.

In order to prove our claim, we now assume that $P$ is $G$-pure with
no multiple factors.  Without loss of generality, we may further
assume that $Q$ is non-constant.  Write $A=\{a_1,\dots,a_s\}$, and
for each $\ux = (x_1,\dots,x_s) \in \bZ^s$ define $a^\ux :=
a_1^{x_1}\cdots a_s^{x_s}$.  With this notation, we have
$a_i=a^{\ue_i}$ for $i=1,\dots,s$ where $(\ue_1, \dots, \ue_s)$
denotes the canonical basis of $\bZ^s$. Moreover, since
$a_1,\dots,a_s$ are multiplicatively independent, the map from
$\bZ^s$ to $G$ which sends each $\ux\in\bZ^s$ to $a^\ux\in G$ is a
group isomorphism.  Choose an irreducible factor $R$ of $P$. As
$P(0)\neq 0$, the polynomial $R$ is not a rational multiple of $T$,
and so Lemma \ref{trans:lemmaH2} shows that, for distinct points
$\ux,\uy\in\bZ^s$, the translates $R(a^{-\ux}T)$ and $R(a^{-\uy}T)$
are not associates. Therefore $P$ is an associate of $\prod_{\ux\in
F} R(a^{-\ux}T)$ for a unique finite subset $F$ of $\bZ^s$, and we
find
\[
 Q = \gcd\Big\{
         \prod_{\ux\in F} R(a^{-\ux} a_i T) \ ;\  i=1,\dots,s
         \Big\}
   = \prod_{\ux\in E} R(a^{-\ux}T),
\]
where $E=(F-\ue_1)\cap \cdots \cap (F-\ue_s)$.  We now apply
Proposition \ref{inter:propC2} to the sets $E$ and $F$.  Since $Q$
is non-constant, the set $E$ is not empty and this proposition
provides an integer $r\ge 1$, a sequence of points
$\ux_1,\dots,\ux_r$ of $E$, and partitions
\[
 E = E_1 \amalg\cdots\amalg E_r
 \et
 F = F_1 \amalg\cdots\amalg F_r \amalg F_{r+1}
\]
of $E$ and $F$ which, for $i=1,\dots,r$, satisfy $E_i \subseteq
C_\ell(\ux_i)$, $F_i \subseteq \cO(E_i)$ and $|E_i| \le \rho |F_i|$
where $\rho=2(\ell+1)/(s-\ell)$.   The third set of conditions
implies $|E| \le \rho |F|$.  Since for each $\ux\in\bZ^s$ the
polynomial $R(a^{-\ux}T)$ has the same degree as $R$, we deduce that
\[
 \deg(Q) = |E| \deg(R) \le \rho |F| \deg(R) = \rho \deg(P).
\]
To compare the heights of $Q$ and $P$, we put $R_i = R(a^{-\ux_i}T)$
for $i=1,\dots,r$. The condition $E_i \subseteq C_\ell(\ux_i)$
implies that, for each $\ux\in E_i$, we have $\|\ux-\ux_i\|_1 \le
2\ell$ and so
\[
 \log H(a^{\ux-\ux_i})
  \le 2\ell \max_{1\le k\le s} \log H(a_k)
    = 2\ell \log(c_A).
\]
Since $R(a^{-\ux}T) = R_i(a^{\ux_i-\ux}T)$, Lemma
\ref{trans:lemmaH1} then gives
\[
 | \log H(R(a^{-\ux}T)) - \log H(R_i) |
 \le \log(3 H(a^{\ux-\ux_i}))\deg(R)
 \le (2+2\ell \log(c_A)) \deg(R)
\]
for each $\ux\in E_i$.  Since $Q = \prod_{i=1}^r \prod_{\ux\in E_i}
R(a^{-\ux}T)$, we deduce that
\begin{align*}
 \Big| \log H(Q) - \sum_{i=1}^r |E_i| \log H(R_i) \Big|
 &\le \deg(Q) + \sum_{i=1}^r \sum_{\ux\in E_i}|
      \log H(R(a^{-\ux}T)) - \log H(R_i) | \\
 &\le \deg(Q) + (2+2\ell \log(c_A)) |E| \deg(R) \\
 &\le (3+2\ell \log(c_A)) \deg(Q).
\end{align*}
The condition $F_i \subseteq \cO(E_i)$ in turn implies that, for
each $\ux\in F_i$, we have $\|\ux-\ux_i\|_1 \le 2\ell+1$ and so the
same computations lead to
\[
 \Big| \log H(P) - \sum_{i=1}^{r+1} |F_i| \log H(R_i) \Big|
 \le (3+(2\ell+1)\log(c_A)) \deg(P).
\]
Putting all these estimates together we conclude finally that
\begin{align*}
 \log H(Q)
 &\le (3+2\ell \log(c_A)) \deg(Q) + \sum_{i=1}^r |E_i| \log H(R_i) \\
 &\le \rho \Big( (3+2\ell \log(c_A)) \deg(P) + \sum_{i=1}^{r+1} |F_i| \log H(R_i)
          \Big) \\
 &\le \rho \big( c_2 \deg(P) + \log H(P) \big).
\end{align*}
where $c_2=6+(4\ell+1)\log(c_A)$.
\end{proof}

%
%

\section{Further small value estimates}
\label{sec:further}

The next result refines Proposition \ref{basic:propR} in a context
where the estimates of the preceding section apply.  We use it below
to prove Part 5) of Theorem \ref{thm:main} in a general form
involving a subgroup of arbitrary rank.

\begin{proposition}
 \label{further:propRbis}
Let $\ell\ge 0$ and $n,t\ge 1$ be integers.  Let $A$ be a finite
multiplicatively independent subset of $\Qplus$, let $s=|A|$ denote
its cardinality,  and let $E$ be a finite non-empty subset of
$\Cmult$.  Assume that
\begin{equation}
 \label{propRbis:eqprop0}
 s \ge \max\{\ell+2,2\ell\}
 \et
 \max(s,|E|) t
 \le n
 \le N(s,\ell):=\binom{s}{\ell+2}\frac{1}{(\ell+1)!2^{\ell+1}}.
\end{equation}
Finally, let $X$ be a real number satisfying
\begin{equation}
 \label{propRbis:eqprop1}
 X^\epsilon
 \ge
 \max\{ 3^n, c_A^n, (2+c_E)^n, \delta_E^{-|E|^2t^2/n} \},
\end{equation}
where $\epsilon=(4\ell+10)^{-1}$, $c_A=\max_{a\in A} H(a)$ and
$c_E=\max_{\xi\in E} \max\{ |\xi|, |\xi|^{-1}\}$.  Suppose that
there exists a non-zero polynomial $P$ of $\ZT$ of degree at most
$n$ and height at most $X$ satisfying
\begin{equation}
 \label{propRbis:eqprop2}
 \max\{|P^{[j]}(a\xi)| \,;\, a\in A,\, \xi\in E,\, 0\le j< 2t \}
 \le
 X^{-\kappa n/(|E|t)}
\end{equation}
for some real number $\kappa > 2+34\epsilon$. Then, there exists a
primary polynomial $S\in\ZT$ with
\begin{equation}
 \label{propRbis:eqprop3}
 \deg(S)\le \frac{2n}{st},
 \quad
 H(S) \le X^{4/(st)}
 \et
 \prod_{\xi\in E} |S(\xi)| \le X^{-\kappa'n/t^2},
\end{equation}
where $\kappa'=(\kappa-2-34\epsilon)/(64(\ell+1))$.
\end{proposition}

In the applications that we will make of this result, the
cardinality $s$ of $A$ is bounded below by $n^\sigma$ for some real
number $\sigma>0$, and so the condition $n\le N(s,\ell)$ is
satisfied with $\ell=[1/\sigma]$ provided that $n$ is large enough.

\begin{proof}
Write $P$ in the form $P(T) = T^m \tP(T)$ where $\tP(T)\in\ZT$ is
not divisible by $T$.  Then, $\tP$ also has degree at most $n$ and
height at most $X$. Moreover, for any $a\in A$, any $\xi\in E$ and
any integer $j$ with $0\le j <2t$, we find
\begin{align*}
 |\tP^{[j]}(a\xi)|
 & = \left| \sum_{h=0}^j (-1)^h \binom{m+h-1}{h} (a\xi)^{-m-h}
     P^{[j-h]}(a\xi) \right| \\
 & \le \sum_{h=0}^j 2^{m+h-1}
       \max\big( 1, H(a)|\xi|^{-1} \big)^{m+2t}
       \max_{0\le i < 2t} |P^{[i]}(a\xi)| \\
 & \le (2c_Ac_E)^{m+2t} X^{-\kappa n/(t |E|)}.
\end{align*}
Since $n\ge t|E|$ and $(2c_Ac_E)^{m+2t} \le (2c_Ac_E)^{3n} \le
X^{9\epsilon}$, we conclude that
\begin{equation}
 \label{propRbis:eq1}
 \max\{|\tP^{[j]}(a\xi)| \,;\, a\in A,\, \xi\in E,\, 0\le j< 2t \}
 \le
 X^{-(\kappa-9\epsilon) n/(t |E|)}.
\end{equation}

Let $\tQ$ be the greatest common divisor in $\QT$ of the polynomials
$\tP(aT)$ with $a\in A$.  Since $A$ is a multiplicatively
independent subset of $\Qplus$, since $\tP(0)\neq 0$, and since the
number of irreducible factors of $\tP$ is at most $\deg(\tP) \le n
\le N(s,\ell)$, Theorem \ref{gcd:thmG} gives
\begin{align*}
 \deg(\tQ)
 & \le \frac{2(\ell+1)}{s-\ell} \deg(\tP) \le 4(\ell+1)\frac{n}{s} \\
 \noalign{\noindent \text{and}}
 \log H(\tQ)
 &\le \frac{2(\ell+1)}{s-\ell} \big( \log H(\tP)
      + (8 + (4\ell+1)\log(c_A)) \deg(\tP) \big) \\
 &\le \frac{4(\ell+1)}{s} \big( \log(X) + 8n + (4\ell+1)\log(c_A)n \big) \\
 &\le 4(\ell+1) ( 2- \epsilon ) \frac{\log X}{s}
\end{align*}
where the last estimation uses $\max(1,\log c_A) n \le \epsilon \log
X$ and $(4\ell+10)\epsilon = 1$.

Let $Q$ be the greatest common divisor in $\QT$ of the polynomials
$\tP^{[j]}(aT) = \lambda_{a,0}(P^{[j]})(T)$ with $a\in A$ and $0\le
j <t$. Since $Q$ divides $\tQ$, we have
\begin{equation*}
 \deg(Q) \le 4(\ell+1)\frac{n}{s}
 \et
 \log H(Q)
  \le \deg(\tQ) + \log H(\tQ)
  \le 8(\ell+1) \frac{\log X}{s}.
\end{equation*}
Moreover, Proposition \ref{basic:propQ} applied to $\tP$ gives
\begin{equation*}
 \prod_{\xi\in E} \left(\frac{|Q(\xi)|}{\cont(Q))}\right)^t
  \le X^{25\epsilon n} H(\tP)^{2n} X^{-(\kappa-9\epsilon) n}
  \le X^{-(\kappa-2-34\epsilon)n}
  \le X^{-64(\ell+1)\kappa' n}.
\end{equation*}
This means that $Q$ satisfies the hypotheses of Lemma
\ref{lemma:linearization} with $\rho=8(\ell+1)/s$ and
$c=32(\ell+1)\kappa'$.  Consequently, there is at least one
irreducible factor $R$ of $Q$ in $\QT$ which satisfies
\begin{equation}
 \label{propRbis:eq3}
 \prod_{\xi\in E} \left(\frac{|R(\xi)|}{\cont(R)}\right)^t
 \le
 \big( X^{\deg(R)} H(R)^n \big)^{-2\kappa's}.
\end{equation}
By Lemma \ref{trans:lemmaH3}, this polynomial also satisfies
\[
 \deg(R) \le \frac{n}{st}
 \et
 H(R) \le \big( (3c_A)^{2n} H(\tP) \big)^{1/(st)}
      \le X^{(1+4\epsilon)/(st)}
      \le X^{2/(st)}.
\]
Applying Lemma \ref{lemma:linearization} to $R$ with $\rho = 2/(st)$
and $c = 2\rho\kappa's = 4\kappa'/t$, we deduce that some power $S$
of $R$ satisfies
\[
 \deg(S) \le \frac{2n}{st},
 \quad
 H(S) \le X^{4/(st)}
 \et
 \prod_{\xi\in E} \left(\frac{|S(\xi)|}{\cont(S)}\right)^t
 \le
 X^{-\kappa'n/t}.
\]
The quotient of $S$ by its content is then a (non-constant) primary
polynomial of $\bZ[T]$ with the required properties
\eqref{propRbis:eqprop3}.
\end{proof}

For $m=2$, the following result reduces to Part 5) of Theorem
\ref{thm:main}.

\begin{theorem}
 \label{further:thm}
Let $\xi_1,\dots,\xi_m$ be $\bQ$-linearly independent complex
numbers which generate a field of transcendence degree one over
$\bQ$.  Let $\beta, \sigma, \tau, \nu \in \bR$ with
\[
 \sigma\ge 0,\quad
 \tau\ge 0, \quad
 \beta > 1 > \frac{3m\sigma}{m+2} + \tau
 \et
 \nu > 1 + \beta - \frac{2m\sigma}{m+2} - \tau.
\]
Then, for infinitely many integers $n\ge 1$, there is no non-zero
polynomial $P\in\ZT$ of degree at most $n$ and height at most
$\exp(n^\beta)$ which satisfies
\[
 |P^{[j]}(i_1\xi_1+\cdots+i_m\xi_m)| \le \exp(-n^\nu)
\]
for each choice of integers $i_1,\dots,i_m,j$ with $0\le
i_1,\dots,i_m \le n^\sigma$ and $0\le j < n^\tau$.
\end{theorem}

\begin{proof}
Suppose on the contrary that such a polynomial exists for each
sufficiently large value of $n$.  Then we have $\sigma>0$ by
\cite[Prop.~1]{LR1}.  Moreover, since $\xi_1,\dots,\xi_m$ are not
all algebraic over $\bQ$, we may assume without loss of generality
that $\xi_1$ is transcendental over $\bQ$.  Define
\[
 \lambda = \frac{\sigma}{m+2},
 \quad
 \ell = \Big[\frac{2}{m\lambda}\Big],
 \quad
 \delta =
  \frac{1}{8} \min\big\{ 4m\lambda,\, 1 - 3m\lambda - \tau,\,
      \nu - 1 - \beta + 2m\lambda + \tau \big\},
\]
and note that the hypotheses lead to $\delta>0$.

For a given positive integer $n$, define $A$ to be the set of all
prime numbers $p$ with $p \le n^{m\lambda}$, and define $E$ to be
the set of all linear combinations $i_1\xi_1+\cdots+i_m\xi_m$ with
integer coefficients in the range $1 \le i_1,\dots,i_m \le
n^{2\lambda}$, which are not algebraic over $\bQ$ and have absolute
value at least $1$. Since for fixed $i_2,\dots,i_m$, there are at
most $1+2/|\xi_1|$ values of $i_1$ for which $i_1\xi_1 + \cdots +
i_m\xi_m$ has absolute value less than one, and at most one value of
$i_1$ for which it is algebraic over $\bQ$, we readily get that, for
$n$ sufficiently large, we have
\[
 n^{m\lambda-\delta} \le |A| \le n^{m\lambda}
 \et
 n^{2m\lambda-\delta} \le |E| \le n^{2m\lambda}.
\]

Suppose first that $\delta_E \ge \exp(
-n^{1+\beta-4m\lambda-2\tau-\delta} )$.  Then, we claim that, if $n$
is sufficiently large, all the hypotheses of Proposition
\ref{further:propRbis} are satisfied with the choice of
\[
 s = |A|,
 \quad
 t=[ (1 + n^\tau)/2 ],
 \quad
 X=\exp(n^\beta),
 \et
 \kappa=n^{6\delta}.
\]
First of all, we have $\max(s,|E|)t \le n$ because $2m\lambda + \tau
< 1$.  As $\delta\le m\lambda/2$, we have $s \ge n^{m\lambda/2}$ and
so $N(s,\ell) \ge n$ for $n$ large enough.  For large $n$, we also
find $c_A\le n$, $c_E \le n$ and $\delta_E^{-|E|^2t^2/n} \le
\exp(n^{\beta-\delta}) \le X^\epsilon$ with $\epsilon =
(4\ell+10)^{-1}$, while the hypothesis on $P$ gives
\[
 \max\{|P^{[j]}(a\xi)| \,;\, a\in A, \, \xi\in E,\, 0\le j< 2t \}
  \le \exp(-n^\nu)
  \le X^{-\kappa n/(t|E|)}.
\]
Consequently, for each sufficiently large value of $n$, there exists
$S\in\ZT\setminus\{0\}$ with
\[
 \deg(S)
   \le \frac{2n}{st},
 \quad
 H(S)
   \le \exp\left( \frac{4n^\beta}{st} \right)
 \et
 \prod_{\xi\in E} |S(\xi)| \le \exp(-n^{1+\beta-2\tau+5\delta}).
\]
Since $|E| \le n^{2m\lambda}$, the last condition implies the
existence of a point $\xi\in E$ such that
\[
 |S(\xi)|
 \le \exp(-n^{1+\beta-2m\lambda-2\tau+5\delta}).
\]
Moreover, since $\xi$ is transcendental over $\bQ$, we have
$S(\xi)\neq 0$.  Define
\[
 Q(T_1,\dots,T_m)=S(i_1T_1+\cdots+i_mT_m) \in \bZ[T_1,\dots,T_m],
\]
where $i_1,\dots,i_m$ are the positive integers for which
$\xi=i_1\xi_1+\cdots+i_m\xi_m$.  Since $i_1,\dots,i_m$ are bounded
above by $n^{2\lambda}$, we find, assuming that $n$ is sufficiently
large,
\begin{align*}
 \deg(Q)
 \le m\deg(S) &\le n^{1-m\lambda-\tau+2\delta}, \\
 H(Q)
 \le (1+mn^{2\lambda})^{\deg(S)}H(S)
  &\le \exp(n^{\beta-m\lambda-\tau+2\delta}),\\
 0 < |Q(\xi_1,\dots,\xi_m)|
   = |S(\xi)|
    &\le \exp(-n^{1+\beta-2m\lambda-2\tau+5\delta}).
\end{align*}

Suppose now that $\delta_E < \exp( -n^{1 + \beta - 4m\lambda - 2\tau
- \delta} )$, and choose integers $i_1,\dots,i_m$ not all zero, in
absolute value at most $n^{2\lambda}$, such that
$|i_1\xi_1+\cdots+i_m\xi_m| = \delta_E$. Then, if $n$ is large
enough, the polynomial $Q=(i_1T_1+\cdots+i_mT_m)^{[2n/(st)]}$
satisfies the same final estimates as in the preceding case, because
$\delta_E^{[2n/(st)]} \le \exp( -n^{2 + \beta - 5m\lambda - 3\tau -
\delta} ) \le \exp( -n^{1 + \beta - 2m\lambda - 2\tau + 7\delta} )$.
The existence of such a polynomial $Q$ for each $n$ large enough
contradicts Lemma \ref{gelfond:curve}.
\end{proof}

%
%

\section{A note on Zarankiewicz problem}
\label{sec:Zaran}

Given integers $m_1,n_1,m,n$ with $2\le m_1\le m$ and $2\le n_1\le
n$, a well-known problem of K.~Zarankiewicz asks for the smallest
integer $k=k(m_1,n_1;m,n)$ such that any $m\times n$ matrix with
coefficients in $\{0,1\}$ containing at least $k$ ones admits a
sub-matrix of size $m_1\times n_1$ consisting only of ones.  Chapter
12 of the book \cite{ES} by P.~Erd\"os and J.~Spencer provides
general estimates for this quantity along with references to early
work on this problem.  In particular, we mention a result of
T.~K\"ovari, V.~T.~S\'os and P.~Tur\'an \cite{KST} which shows that
$k(2,2;n,n)=n^{3/2}(1-o(1))$.  In the next section, we will use the
following result which we view as an estimate for a continuous
version of Zarankiewicz problem in the case $m_1=2$.

\begin{proposition}
 \label{Zaran:propZ}
Let $A$ and $E$ be finite non-empty sets, let $\kappa_1$ and
$\kappa_2$ be positive real numbers, and let $\varphi\colon A\times
E \to [0,\kappa_1]$ be any function on $A\times E$ with values in
the interval $[0,\kappa_1]$. Suppose that the inequality
\[
 \sum_{\xi\in E} \min\{\varphi(a_1,\xi),\varphi(a_2,\xi)\} \le
 \kappa_2
\]
holds for any pair of distinct elements $a_1$ and $a_2$ of $A$.
Then, we have
\[
 \sum_{a\in A}\sum_{\xi\in E} \varphi(a,\xi)
 \le
 \kappa_1 |E| + \kappa_2 \binom{|A|}{2}.
\]
\end{proposition}

This gives $k(2,n_1;m,n)\le 1+n+(n_1-1)m(m-1)/2$ in connection to
the problem of Zarankiewicz mentioned above.  Indeed, an $m\times n$
matrix with coefficients in $\{0,1\}$ can be viewed as a function
$\varphi\colon A\times E \to \{0,1\}$ where $A=\{1,\dots,m\}$ and
$E=\{1,\dots,n\}$.  If it contains no $2\times n_1$ sub-matrix
consisting entirely of ones, the hypotheses of the proposition are
satisfied with $\kappa_1=1$ and $\kappa_2=n_1-1$ and consequently
the matrix contains at most $n+(n_1-1)m(m-1)/2$ ones.

\begin{proof}
Let $[0,\infty)^E$ denote the set of all functions from $E$ to
$[0,\infty)$.  We first observe that, if $\phi$ and $\psi$ belong to
this set then their minimum, their maximum and their sum
\begin{equation}
 \label{propZ:eq1}
 \phi + \psi
  = \min(\phi,\psi) + \max(\phi,\psi)
\end{equation}
also belong to it.  Moreover, the $\ell_1$-norm of any function
$\phi\colon E \to [0,\infty)$ is simply $\|\phi\|_1 = \sum_{\xi\in
E} \phi(\xi)$. Therefore the $\ell_1$-norm is additive on
$[0,\infty)^E$ and by applying it on both sides of the equality
\eqref{propZ:eq1} with functions $\phi,\psi\in [0,\infty)^E$, we
obtain
\begin{equation}
 \label{propZ:eq2}
 \|\phi\|_1 +\|\psi\|_1
  = \|\min\{\phi,\psi\}\|_1 + \|\max\{\phi,\psi\}\|_1.
\end{equation}

Let $m=|A|$ and let $a_1,\dots,a_m$ denote the $m$ elements of $A$.
For $i=1,\dots,m$, we define a function $\phi_i\colon E\to
[0,\kappa_1]$ by putting $\phi_i(\xi)=\varphi(a_i,\xi)$ for each
$\xi\in E$.  By hypothesis, we have $\| \min(\phi_i, \phi_j) \|_1
\le \kappa_2$ for any pair of integers $i$ and $j$ with $1\le i<j\le
m$.  So, for $j=2,\dots,m$, the function
$
 \min( \max\{\phi_1, \dots, \phi_{j-1}\}, \phi_j )
 = \max\{ \min(\phi_1,\phi_j), \dots, \min(\phi_{j-1},\phi_j) \}
$
satisfies
\[
 \|\min( \max\{\phi_1, \dots, \phi_{j-1}\}, \phi_j )\|_1
 \le \sum_{i=1}^{j-1} \| \min(\phi_i,\phi_j) \|_1
 \le (j-1)\kappa_2.
\]
Applying \eqref{propZ:eq2} with $\phi=\max\{\phi_1, \dots,
\phi_{j-1}\}$ and $\psi=\phi_j$, we deduce that
\[
 \| \max\{\phi_1, \dots, \phi_{j-1}\} \|_1 + \|\phi_j\|_1
 \le
 (j-1)\kappa_2 + \| \max\{\phi_1, \dots, \phi_j\} \|_1.
\]
Summing these inequalities term by term for $j=2,\dots,m$, we obtain
after simplification
\[
 \sum_{j=1}^m \|\phi_j\|_1
 \le
 \sum_{j=2}^m (j-1)\kappa_2 + \| \max\{\phi_1, \dots, \phi_m\}
 \|_1.
\]
Since $\max\{\phi_1, \dots, \phi_m\}$ takes values in
$[0,\kappa_1]$, its $\ell_1$-norm is at most $\kappa_1 |E|$, and the
conclusion follows upon noting that $\sum_{a\in A}\sum_{\xi\in E}
\varphi(a,\xi) = \sum_{j=1}^m \|\phi_j\|_1$.
\end{proof}

%
%

\section{Values of polynomials at multiples of $\xi$}
\label{sec:multxi}

In this section, we first prove Theorem \ref{intro:thm:prod} as a
consequence of the next proposition, and then proceed with the proof
of Theorem \ref{thm:main}, part 3).

\begin{proposition}
 \label{multxi:propRter}
Let $n\in \bN^*$, let $A$ be a finite subset of $\Qplus$, and let
$E$ be a finite subset of $\Cmult$. Assume that
\begin{equation}
 \label{propRter:eqprop1}
 |A| \ge 2
 \et
 \binom{|A|}{2} \le |E| \le n.
\end{equation}
Finally, let $\epsilon$ and $X$ be real numbers with
\begin{equation}
 \label{propRter:eqprop2}
 0< \epsilon \le \frac{1}{10}
 \et
 X^\epsilon
 \ge
 \max\{ e^n, c_A^n, (2+c_E)^n, \Delta_E^{-1/n} \},
\end{equation}
where $c_A=\max_{a\in A} H(a)$ and $c_E=\max_{\xi\in E}
\max(|\xi|,|\xi|^{-1})$. Suppose that there exists a non-zero
polynomial $P$ of $\ZT$ of degree at most $n$ and height at most $X$
satisfying
\begin{equation}
 \label{propRter:eqprop3}
 \prod_{a\in A} \prod_{\xi\in E} |P(a\xi)|
 < X^{-16 \kappa |E| n}
\end{equation}
for some real number $\kappa \ge 6$. Then, there exist a primary
polynomial $S\in\QT$ and a point $\xi\in E$ satisfying
\[
 \deg(S)\le n,
 \quad
 H(S) \le X^{2+2\epsilon}
 \et
 \frac{|S(\xi)|}{\|S\|} \le X^{-\kappa n}.
\]
\end{proposition}

\begin{proof}
Applying the linearization Lemma \ref{lemma:linearization} b) with
$\rho=1$ and $c=8\kappa|E|$, we find that there exists a power $Q$
of some non-constant irreducible factor of $P$ in $\ZT$ satisfying
\begin{equation}
 \label{propRter:eq3}
 \deg(Q)\le n,
 \quad
 H(Q)\le X^2
 \et
 \prod_{a\in A}\prod_{\xi\in E} |Q(a\xi)|
 \le
 X^{-2\kappa |E| n}.
\end{equation}
We also note that $Q$ is not a power of $T$ because, for each $a\in
A$ and each $\xi\in E$, we have $|a|\ge c_A^{-1} \ge
X^{-\epsilon/n}$ and $|\xi|\ge c_E^{-1} \ge X^{-\epsilon/n}$ and so,
for a power of $T$, the product $\prod_{a\in A}\prod_{\xi\in E}
|Q(a\xi)|$ would be bounded below by $X^{-2\epsilon|A||E|} \ge
X^{-4\epsilon |E| n}$ against the upper bound.  For each $a\in A$
and each $\xi\in E$, we find
\[
 \|Q(aT)\| \ge H(a)^{-n}\|Q\| \ge c_A^{-n} \ge X^{-\epsilon}
 \et
 \frac{|Q(a\xi)|}{\|Q(aT)\|}
   \le (|\xi|+1)^n \le (c_E+1)^n \le X^\epsilon.
\]
Therefore we can write
\begin{equation}
 \label{propRter:eq5}
 \min(1,|Q(a\xi)|)
 \ge \frac{|Q(a\xi)|}{X^\epsilon \|Q(aT)\|}
 = X^{-\varphi(a,\xi)n}
\end{equation}
for some real number $\varphi(a,\xi)\ge 0$.   This define a function
$\varphi\colon A\times E \to [0,\infty)$ which, by the last
condition of \eqref{propRter:eq3}, satisfies
\begin{equation}
 \label{propRter:eq6}
 \sum_{a\in A} \sum_{\xi\in E} \varphi(a,\xi)
 \ge
 2\kappa |E|.
\end{equation}
Moreover, for each $a\in A$, Lemma \ref{trans:lemmaH1} gives
\begin{equation}
 \label{propRter:eq4}
 H(Q(aT)) \le (3H(a))^n H(Q) \le (3c_A)^n X^2 \le X^{2+2\epsilon}.
\end{equation}
We claim that we have $\varphi(a,\xi) > \kappa + \epsilon$ for at
least one choice of $a\in A$ and $\xi\in E$.  If we admit this
result, then for such choice of $a$ and $\xi$ the polynomial $S(T) =
Q(aT)$ and the point $\xi$ have all the required properties.  First
of all, Lemma \ref{trans:lemmaH2} shows that $S$ is, like $Q$, a
primary polynomial of $\QT$.  Its degree is at most $n$ and by
\eqref{propRter:eq4} its height at most $X^{2+2\epsilon}$. Finally,
by definition of $\varphi(a,\xi)$, we also have
\[
 \frac{|S(\xi)|}{\|S\|}
 = \frac{|Q(a\xi)|}{\|Q(aT)\|}
 \le X^{\epsilon-\varphi(a,\xi)n}
 \le X^{-\kappa n}.
\]

To prove our claim, we proceed by contradiction assuming on the
contrary that $\varphi$ takes values in $[0,\kappa+\epsilon]$. Let
$a_1$ and $a_2$ be two distinct elements of $A$.  We apply
Proposition \ref{result:propFG} to the polynomials $Q(a_1T)$ and
$Q(a_2T)$ with $s=|E|$ and $t=1$.  Since $Q$ is primary and not a
power of $T$, and since $a_1/a_2 \neq \pm 1$, Lemma
\ref{trans:lemmaH2} shows that these polynomials are relatively
prime in $\QT$. Therefore, the proposition gives
\begin{equation}
 \label{propRter:eq7}
 1
 \le
 c H(Q(a_1T))^n H(Q(a_2T))^n
   \prod_{\xi\in E}
   \max\left\{ \frac{|Q(a_1\xi)|}{\|Q(a_1T)\|},
               \frac{|Q(a_2\xi)|}{\|Q(a_2T)\|} \right\}
\end{equation}
where $c = e^{7n^2} (c_E+2)^{4|E|n} \Delta_E^{-1}$. Using $|E|\le n$
and the hypotheses \eqref{propRter:eqprop1}, we find $c\le
X^{12\epsilon n}$. Substituting this into \eqref{propRter:eq7} and
using \eqref{propRter:eq4} and \eqref{propRter:eq5}, we find
\[
 1
 \le
 X^{12\epsilon n}
 ( X^{2+2\epsilon} )^{2n}
 \prod_{\xi\in E} X^{ \epsilon - n \min\{\varphi(a_1,\xi),\varphi(a_2,\xi)\}},
\]
and therefore, since $|E|\le n$ and $\epsilon \le 1/10$, we finally
obtain
\[
 \sum_{\xi\in E} \min\{\varphi(a_1,\xi),\varphi(a_2,\xi)\}
 \le 6-3\epsilon.
\]
According to Proposition \ref{Zaran:propZ}, this implies that
\[
 \sum_{a\in A} \sum_{\xi\in E} \varphi(a,\xi)
 \le (6-3\epsilon) \binom{|A|}{2} + (\kappa+\epsilon) |E|
 \le (6+\kappa-2\epsilon) |E|
 < 2\kappa |E|,
\]
in contradiction with \eqref{propRter:eq6}.  Therefore $\varphi$
must take at least one value greater than $\kappa+\epsilon$.
\end{proof}

\begin{proof}[Proof of Theorem \ref{intro:thm:prod}]
It suffices to prove the result in the case where $\alpha=1$. We
proceed by contradiction, assuming on the contrary that for each
sufficiently large $n$ there exists a non-zero polynomial
$P_n\in\ZT$ of degree at most $n$ and height at most $\exp(n^\beta)$
satisfying $\prod_{a\in A_n}\prod_{b\in B_n} |P_n(ab\xi)| \le
\exp(-n^{1+\beta+2\mu+\delta})$.  We claim that, for $n$ large
enough, all the hypotheses of Proposition \ref{multxi:propRter} are
satisfied with the choice of $A=A_n$, $E=B_n\xi$, $\epsilon=1/10$,
$X=\exp(n^\beta)$, $P=P_n$ and $\kappa=n^{\delta}$.  First of all
the conditions \eqref{propRter:eqprop1} and \eqref{propRter:eqprop3}
are fulfilled because the prime number theorem shows that the
cardinalities of $A_n$ and $B_n$ behave respectively like
$n^\mu/(\mu\log n)$ and $n^{2\mu}/(2\mu\log n)$, and we have
$2\mu<1<\beta$. Finally the condition \eqref{propRter:eqprop2} is
also satisfied as we have $c_A \le n$, $c_E\le n$ and $\Delta_E\ge
1$ (for $n$ large enough). Therefore, there exist a point $b\in B_n$
and a non-zero primary polynomial $S$ of $\QT$ with
\[
 \deg(S) \le n,
 \quad
 H(S)\le \exp(3n^\beta)
 \et
 \frac{|S(b\xi)|}{\|S\|} \le \exp(-n^{1+\beta+\delta}).
\]
Upon dividing $S$ by its content, we may assume that $S\in\ZT$.
Then, assuming again that $n$ is sufficiently large, we deduce that
the polynomial $Q=S(bT)\in\ZT$ satisfies
\[
 \deg(Q) \le n,
 \quad
 H(Q)\le \exp(n^{\beta+\delta/2})
 \et
 |Q(\xi)| \le \exp(-n^{1+\beta+\delta}),
\]
against Gel'fond's Lemma \ref{gelfond:curve}.
\end{proof}

\begin{proof}[Proof of Theorem \ref{thm:main}, part 3)]
Suppose on the contrary that, for each sufficiently large $n$, the
polynomial $P_n$ satisfies $|P_n^{[j]}(i\xi)| \le \exp(-n^\nu)$ for
each choice of integers $i$ and $j$ with $1\le i\le n^\sigma$ and
$0\le j\le n^\tau$.  By \cite[Prop.~1]{LR1}, we must have
$\sigma>0$. Define
\[
 \ell = [4/\sigma]
 \et
 \delta = (1/5)\big(\nu - 1 - \beta + (3/4)\sigma + \tau \big).
\]
For a given integer $n\ge 1$, let $A=A_n$ be the set of all prime
numbers $p$ with $p \le n^{\sigma/4}$, and let $E=AB\xi$ where
$B=B_n$ is the set of all prime numbers $p$ with $n^{\sigma/4} < p
\le n^{\sigma/2}$. We claim that if $n$ is sufficiently large, all
the hypotheses of Proposition \ref{further:propRbis} are satisfied
with the additional choice of
\[
 t=[(1+n^\tau)/2],
 \quad
 X=\exp(n^\beta),
 \quad
 P=P_n
 \et
 \kappa=n^{4\delta}.
\]
The conditions \eqref{propRbis:eqprop0} are fulfilled because we
have $(3/4)\sigma+\tau<1$ and for large enough values of $n$ the
prime number theorem gives
\[
 n^{\sigma/4}(\log n)^{-1} \le |A| \le n^{\sigma/4}
 \et
 n^{3\sigma/4}(\log n)^{-2} \le |E| \le n^{3\sigma/4}.
\]
The conditions \eqref{propRbis:eqprop1} are also satisfied because
we have $\beta>1$ and for large enough values of $n$ we find $c_A\le
n$, $c_E \le n$ and $\delta_E^{-|E|^2t^2/n} \le \min(1,|\xi|)^{-n}
\le X^\epsilon$ (since $|E|t\le n$). Finally, as the product $AE$ is
contained in $\{\xi,2\xi,\dots,[n^\sigma]\xi\}$, the hypothesis on
$P$ gives
\[
 \max\{|P^{[j]}(ax)| \,;\, a\in A,\, x\in E,\, 0\le j< 2t \}
  \le \exp(-n^\nu)
  \le X^{-\kappa n/(t|E|)},
\]
and so the main condition \eqref{propRbis:eqprop2} is also
satisfied. Consequently, for each sufficiently large value of $n$,
there exists a non-zero polynomial $S\in\ZT$ with
\[
\begin{aligned}
 \deg(S)
 &\le \frac{2n}{t |A|} \le n^{1-\sigma/4-\tau+\delta}, \\
 H(S)
 &\le \exp\left(\frac{4n^\beta}{t |A|}\right)
   \le \exp\big(n^{\beta-\sigma/4-\tau+\delta}\big), \\
 \prod_{a\in A} \prod_{b\in B} |S(ab\xi)|
 &\le \exp\big(-n^{1+\beta-2\tau+3\delta}\big),
\end{aligned}
\]
upon noting, for the last inequality, that any element of $E$ can be
written uniquely as a product $ab\xi$ with $a\in A$ and $b\in B$.
This contradicts Theorem \ref{intro:thm:prod} (with $\mu=\sigma/4$).
\end{proof}

%
%

\section{Higher transcendence degree}
\label{sec:higher}

In this section, we prove the part 6) of Theorem \ref{thm:main} by
combining its part 3) with the following result.

\begin{proposition}
 \label{prop:result:E+F}
Let $n,s \in \bN^*$ with $s\le 2n$, let $E$ and $F$ be finite
subsets of $\bC$ with $0\in E$ and $|F|=s$, and let $P$ be any
non-zero polynomial of $\ZT$ of degree at most $n$.  Put
\[
 \delta_P = \max\{ |P(\xi+\eta)|\,;\, \xi\in E,\, \eta\in F\}.
\]
Then there exists a non-zero polynomial $R\in\ZT$ satisfying
\begin{itemize}
 \item[(i)] $\deg(R) \le n^2$,
 \item[(ii)] $H(R) \le 6^{n^2} H(P)^{2n}$,
 \item[(iii)]
 $\max\big\{|R^{[k]}(\xi)|\,;\, \xi\in E,\, 0\le k\le [s/2]\big\}
  \le c H(P)^{2n} \min(1,\delta_P)^{s/2}$,
\end{itemize}
where $c = \Delta_F^{-1} \big(8^n (1+c_E)^n (1+c_F)^s)^{3n}$, $c_E =
\max_{\xi\in E}|\xi|$ and $c_F = \max_{\eta\in F} |\eta|$.
\end{proposition}

\begin{proof}
We may assume without loss of generality that $P$ is primitive.
Suppose first that its degree is $n$.  We claim that the resultant
$R(U)$ of $P(T)$ and $P(T+U)$ with respect to $T$ has the required
properties as a polynomial in the new variable $U$.  Since $P(T)$
and $P(T+U)$ are relatively prime elements of $\bZ[T,U]$, we know
that $R(U)$ is a non-zero polynomial of $\bZ[U]$.  To prove the
estimates (i), (ii) and (iii), we apply Lemma \ref{result:lemmaDet}
with $L=\bQ(U)$, $m=2n$, $t=1$, $Q=1$, the role of $\xi_1, \dots,
\xi_s$ played by the points $\eta_1, \dots, \eta_s$ of $F$, and the
sequence of polynomials $P_1, \dots, P_m$ given by
\begin{equation}
 \label{E+F:seqP}
 P(T),\,TP(T),\dots,T^{n-1}P(T),\,P(T+U),\,TP(T+U),\dots,T^{n-1}P(T+U).
\end{equation}
In the notation of Lemma \ref{result:lemmaDet}, this gives
\begin{equation}
 \label{E+F:eqR}
 R(U) = \det(\psi(P_1),\dots,\psi(P_m))
      = \pm \Delta^{-1} \det(\varphi(P_1),\dots,\varphi(P_m))
\end{equation}
where $\Delta=\prod_{1\le i< j\le s} (\eta_j-\eta_i)$.

To perform the required estimations, we use the following additional
notation.  For each polynomial $G$ in $\bC[U]$ or $\bC[T,U]$, we
denote by $\|G\|_1$ the sum of the absolute values of its
coefficients (its length).  For a row vector $G =(G_1, \dots, G_m)
\in \bC[U]^m$, we denote by $\deg(G)$ the maximum of the degrees of
$G_1,\dots,G_m$ and we put $\|G\|_1 = \|G_1\|_1 + \cdots +
\|G_m\|_1$.  We also define $G^{[k]} = (G_1^{[k]}, \dots,
G_m^{[k]})$ for each integer $k\ge 0$.  We use the same notation for
column vectors.

By definition, $\psi$ maps a polynomial $G\in\bC[T,U]$ with
$\deg_T(G)<m$ to the row vector $\psi(G)\in\bC[U]^m$ formed by its
coefficients as a polynomial in $T$ over the ring $\bC[U]$.  Thus we
have $\deg(\psi(G))=\deg_U(G)$ and $\|\psi(G)\|_1=\|G\|_1$. Applying
this to the representation of $R(U)$ given by \eqref{E+F:eqR} in
terms of $\psi$, we obtain
\begin{align*}
 \deg(R)
  &\le \sum_{i=1}^m \deg \psi(P_i)
    = n \deg_U(P(T+U)) = n^2,\\
 \|R\|_1
  &\le \prod_{i=1}^m \|\psi(P_i)\|_1
    = \|P(T)\|_1^n\, \|P(T+U)\|_1^n
    \le 6^{n^2}H(P)^n,
\end{align*}
where the last step uses the crude estimates $\|P(T)\|_1 \le \|P\|\,
\|(1+T)^n\|_1 = 2^nH(P)$ and similarly $\|P(T+U)\|_1 \le \|P\|\,
\|(1+T+U)^n\|_1 = 3^nH(P)$.  This proves (i) and (ii).

For $j=1,\dots,m$, let $C_j(U)$ denote the $j$-th column of the
$m\times m$ matrix with rows $\varphi(P_1), \dots, \varphi(P_m)$. By
virtue of \eqref{E+F:eqR}, we have $ R(U) = \pm \Delta^{-1}
\det(C_1(U), \dots, C_m(U)) $. Using the multi-linearity of the
resultant, we deduce that for each integer $k\ge 0$ we have
\begin{equation}
 \label{E+F:R[k]}
 R^{[k]}(U)
 = \pm \Delta^{-1}
   \sum_{k_1+\cdots+k_m=k} \det(C_1^{[k_1]}(U), \dots, C_m^{[k_m]}(U))
\end{equation}
where the sum runs through all partitions of $k$ into a sum of $m$
non-negative integers $k_1,\dots,k_m$.

For $j=1,\dots,s$, the transpose of $C_j(U)$ is the row vector
formed by the values at $\eta_j$ of the sequence of polynomials
\eqref{E+F:seqP}:
\[
 {}^tC_j(U)
 = \big( P(\eta_j), \dots, \eta_j^{n-1}P(\eta_j),\,
         P(U+\eta_j), \dots, \eta_j^{n-1}P(U+\eta_j) \big).
\]
For $\xi\in E$, this gives
\begin{equation}
 \label{E+F:C1..s}
 \|C_j(\xi)\|_1
 \le 2(1+|\eta_j|)^n\max(|P(\eta_j)|,|P(\xi+\eta_j)|)
 \le 2(1+c_F)^n\delta_P,
\end{equation}
since both $\eta_j$ and $\xi+\eta_j$ belong to $E+F$ (as $0\in E$).
For each integer $k\ge 1$, we also find
\[
 {}^tC_j^{[k]}(\xi)
 = \big( 0, \dots, 0,\,
         P^{[k]}(\xi+\eta_j), \dots, \eta_j^{n-1}P^{[k]}(\xi+\eta_j)
         \big),
\]
and therefore
\begin{equation}
 \label{E+F:C[k]1..s}
 \begin{aligned}
 \|C_j^{[k]}(\xi)\|_1
 &\le (1+|\eta_j|)^n |P^{[k]}(\xi+\eta_j)| \\
 &\le (1+|\eta_j|)^n (1+|\xi|+|\eta_j|)^n \|P^{[k]}\| \\
 &\le 2^n (1+c_E)^n (1+c_F)^{2n} H(P).
 \end{aligned}
\end{equation}

For $j=s+1,\dots,2n$, the transpose of $C_j(U)$ is the row vector
made of the coefficients of $T^{j-1}$ from the polynomials of the
sequence \eqref{E+F:seqP}.  It is given by
\[
 {}^tC_j(U)
 = \big( P^{[j-1]}(0), \dots, P^{[j-n]}(0),\,
         P^{[j-1]}(U), \dots, P^{[j-n]}(U) \big),
\]
with the convention that $P^{[i]}=0$ when $i<0$.  For each integer
$k\ge 1$, this gives
\[
 {}^tC_j^{[k]}(U)
 = \left( 0, \dots, 0,\,
         \binom{j+k-1}{j-1} P^{[j+k-1]}(U), \dots,
         \binom{j+k-n}{j-n} P^{[j+k-n]}(U) \right),
\]
with the additional convention that the binomial symbol is zero when
its lower entry is negative.  From this we deduce that, for each
$k\ge 0$ and each $\xi\in E$, we have
\begin{equation}
 \label{E+F:C[k]s+1..m}
 \|C_j^{[k]}(\xi)\|_1
 \le n 2^n \max_{0\le i\le n} \max( |P^{[i]}(0)|,\,|P^{[i]}(\xi)| )
 \le 2^{3n} (1+c_E)^n H(P).
\end{equation}

For each integer $k$ with $0\le k\le s/2$  and each partition of $k$
as a sum of non-negative integers $k_1,\dots,k_m$, there are always
at least $s/2$ indices $i$ with $1\le i\le s$ for which $k_i=0$.
Thus, for such $k$ and any $\xi\in E$, the formula \eqref{E+F:R[k]}
combined with \eqref{E+F:C1..s}, \eqref{E+F:C[k]1..s} and
\eqref{E+F:C[k]s+1..m} gives
\begin{align*}
 |R^{[k]}(\xi)|
 &\le \Delta_F^{-1}
      \binom{k+m-1}{m-1}
      \max_{k_1+\cdots+k_m=k} \prod_{j=1}^m \|C_j^{[k_j]}(\xi)\|_1 \\
 &\le \Delta_F^{-1} 2^{3n}
      \big( 2^{3n}(1+c_E)^n \big)^{2n}
      (1+c_F)^{2ns} H(P)^{2n} \min(1,\delta_P)^{s/2}.
\end{align*}
This proves (iii) with $c$ replaced by $c'=\Delta_F^{-1}  2^{9n^2}
(1+c_E)^{2n^2} (1+c_F)^{2ns}$.

In the general case where $P$ has degree $d\le n$, we apply the
preceding estimates to $\tP(T) = T^{n-d}P(T)$.  Since $\tP$ has
degree $n$, same height as $P$, and since it satisfies
\[
 |\tP(\xi+\eta)|
 \le \max(1,|\xi|+|\eta|)^n \delta_P
 \le (1+c_E)^n (1+c_F)^n \delta_P
\]
for any $\xi\in E$ and $\eta\in F$, we conclude that the
corresponding polynomial $R$ satisfies (i), (ii) and (iii) with the
given value of $c$.
\end{proof}

\begin{proof}[Proof of Theorem \ref{thm:main}, part 6)]
Suppose on the contrary that for each sufficiently large $n$, the
polynomial $P_n$ satisfies $|P_n(i\xi+\eta)| \le \exp(-n^\nu)$ for
$i=0,1,\dots,[n^\sigma]$.  If $\sigma=0$, it follows from Lemma
\ref{gelfond:curve} that both $\eta$ and $\xi+\eta$ are algebraic
over $\bQ$. This is impossible since $\xi$ is transcendental over
$\bQ$.  Thus, we have $\sigma>0$ and so there exists $\delta>0$ such
that $\sigma>\delta$ and $\nu> 3 + \beta - (11/4)\sigma + 5\delta$.
We apply Proposition \ref{prop:result:E+F} with $n$ replaced by
$[\sqrt{n}]$,
\[
 P=P_{[\sqrt{n}]}, \quad
 E = \big\{ i\xi\,;\, 0\le i\le 2n^{(\sigma-\delta)/2} \big\}
 \et
 F = E+\eta.
\]
For $n$ sufficiently large, we have $\max(c_E,c_F) \le
n^{\sigma/2}$, $\Delta_F\ge 1$, $2n^{(\sigma-\delta)/2} \le |E| =
|F| \le [\sqrt{n}]$, and $\max\{|P(x+y)|\,;\, x\in E,\, y\in F\}\le
\exp(-(1/2)n^{\nu/2})$.  So, there exists a non-zero polynomial
$R\in\ZT$ with $\deg(R)\le n$, $H(R)\le
\exp(n^{(1+\beta+\delta)/2})$ and
\[
 \max\big\{|R^{[j]}(i\xi)| \,;\,
   0\le i,j\le n^{(\sigma-\delta)/2} \big\}
 \le \exp(-n^{(\nu+\sigma-2\delta)/2}).
\]
This contradicts Theorem \ref{thm:main}, part 3).
\end{proof}

\appendix
\section{Construction of polynomials with given properties}
\label{sec:appendix1}

The following result derives from a simple application of Dirichlet
box principle.

\begin{proposition}
 \label{prop:Dirichlet}
Let $m\in\bN^*$, let $\xi_1,\dots,\xi_m\in\bC$, and let $\beta$,
$\sigma_1, \dots, \sigma_m$, $\tau$, $\nu$ be positive real numbers
with $\sigma_1+\cdots+\sigma_m+\tau<1$ and
$1<\nu<1+\beta-\sigma_1-\cdots-\sigma_m-\tau$.  For each
sufficiently large integer $n\ge 1$, there exists a non-zero
polynomial $P_n\in\ZT$ of degree at most $n$ and height at most
$\exp(n^\beta)$ satisfying $|P_n^{[j]}(i_1\xi_1+\cdots+i_m\xi_m)|\le
\exp(-n^\nu)$ for any choice of integers $i_1,\dots,i_m$ and $j$
with $0\le i_1\le n^{\sigma_1}$, \dots, $0\le i_m\le n^{\sigma_m}$
and $0\le j \le n^\tau$.
\end{proposition}

\begin{proof}
For each sufficiently large integer $n$, the conditions imposed on
the polynomial $P_n$ constitute a system of at most
$2n^{\sigma_1+\cdots+\sigma_m+\tau}$ linear inequations in its $n+1$
unknown coefficients, each having itself complex coefficients of
absolute value at most $2^n ( n^{\sigma_1}|\xi_1| + \cdots +
n^{\sigma_m}|\xi_m| )^n \le \exp(n^\nu)$. The conclusion follows by
applying a generic version of Thue-Siegel lemma like \cite[Lemma
4.12]{Wa2}.
\end{proof}

In the case where $\sigma_1=\cdots=\sigma_m=\sigma$, the main
condition on the parameter $\nu$ in Proposition \ref{prop:Dirichlet}
becomes $\nu < 1+\beta-m\sigma-\tau$.  The next proposition shows
that in some instances, for $m\ge 3$, the weaker condition $\nu <
1+\beta-2\sigma-\tau$ suffices even for a set of $\bQ$-linearly
independent points $\xi_1,\dots,\xi_m$.

\begin{proposition}
  \label{prop:Khintchine}
Let $m\in\bN^*$ with $m\ge 3$ and let $\beta$, $\sigma$, $\tau$,
$\nu$ be positive real numbers with $2\sigma+\tau<1$ and
$1<\nu<1+\beta-2\sigma-\tau$.  There exist $\bQ$-linearly
independent complex numbers $\xi_1,\dots,\xi_m$ with $\xi_1=1$,
which satisfy the following property.  For each sufficiently large
integer $n\ge 1$, there exists a non-zero polynomial $P_n\in\ZT$ of
degree at most $n$ and height at most $\exp(n^\beta)$ satisfying
$|P_n^{[j]}(i_1\xi_1+\cdots+i_m\xi_m)|\le \exp(-n^\nu)$ for any
choice of integers $i_1,\dots,i_m$ and $j$ with $0\le
i_1,\dots,i_m\le n^{\sigma}$ and $0\le j \le n^\tau$.
\end{proposition}

As the proof will show, these examples are ruled out if we assume
that $\xi_1,\dots,\xi_m$ satisfy an appropriate measure of linear
independence over $\bZ$.

\begin{proof}
Without loss of generality, we may assume that $\nu>\beta$.  Choose
$\delta>0$ such that $\delta<\sigma$, $m\delta+2\sigma+\tau < 1$ and
$\nu+\delta < 1+\beta-m\delta-2\sigma-\tau$.  A simple adaptation of
the argument of P.~Philippon in the appendix of \cite{Ph} (based on
a result of Khintchine \cite{K}) provides $\bQ$-linearly independent
complex numbers $\xi_1=1,\xi_2,\dots,\xi_m$ with the property that,
for each integer $n\ge 1$ and each $k=3,\dots,m$, there exist
integers $a_{k,n}$, $b_{k,n}$ and $c_{k,n}$ with
\[
 \max(|a_{k,n}|,|b_{k,n}|) \le |c_{k,n}| \le n^\delta
 \et
 |a_{k,n}+b_{k,n}\xi_2+c_{k,n}\xi_k| \le \exp(-n^{\nu+\delta})
\]
(the choice of the function $\exp(-n^{\nu+\delta})$ is adapted to
our purpose, but any positive valued function of $n\in\bN$ would
work as well; the only new requirement is the condition
$\max(|a_{k,n}|,|b_{k,n}|) \le |c_{k,n}|$ which is easily
fulfilled). By Proposition \ref{prop:Dirichlet}, for each $n$
sufficiently large, there exists a non-zero polynomial $P_n\in\ZT$
of degree at most $n$ and height at most $\exp(n^\beta)$ such that
$|P_n^{[j]}(i_1\xi_1+\cdots+i_m\xi_m)| \le \exp(-n^{\nu+\delta})$
for any choice of integers $i_1,\dots,i_m$ and $j$ with $0\le
i_1,i_2\le n^{\sigma+\delta}$, $0\le i_3,\dots,i_m\le n^{\delta}$
and $0\le j \le n^\tau$.  We claim that this sequence of polynomials
has the required property.  To show this, choose integers $n$,
$i_1,\dots,i_m$ and $j$ with $n\ge 1$, $0\le i_1,\dots,i_m\le
n^{\sigma}$ and $0\le j \le n^\tau$.  After division, one can write
the point $\xi=i_1\xi_1+\cdots+i_m\xi_m$ as a sum $\xi=\xi'+\eta$
where $\xi'=i_1'\xi_1+\cdots+i_m'\xi_m$ and $\eta=\sum_{k=3}^m q_k
(a_{k,n}+b_{k,n}\xi_2+c_{k,n}\xi_k)$ for integers $i'_k$ and $q_k$
satisfying $|i'_k|\le mn^\sigma$ for $k=1,2$, $|i'_k|\le n^\delta$
for $k=3,\dots,m$, and $0\le q_k\le n^{\sigma}$ for $k=3,\dots,m$.
For $n$ large enough, this gives
\begin{align*}
 \big|P_n^{[j]}(\xi)\big|
 &\le
 \big|P_n^{[j]}(\xi')\big| + |\eta|\, \|P_n^{[j]}\|
 (1+|\xi|+|\xi'|)^n \\
 &\le \exp(-n^{\nu+\delta}) +
 (mn^\sigma\exp(-n^{\nu+\delta}))(2^n\exp(n^\beta)) (4mn^\sigma)^n\\
 &\le \exp(-n^\nu).
\end{align*}
\end{proof}

\section{A note on intersection estimates}
\label{sec:appendix2}

We prove the following result as a complement to the estimates of
\S\ref{sec:inter}.

\begin{proposition}
Let $s$ be a positive integer and let $F$ be a non-empty finite
subset of $\bZ^s$. Define $E
=F\cap(F-\ue_1)\cap\cdots\cap(F-\ue_s)$, where $(\ue_1,\dots,\ue_s)$
denote the canonical basis of $\bZ^s$. Then we have $|E| \le |F|-
|F|^{(s-1)/s}$.
\end{proposition}

\begin{proof}
We proceed by induction on $s$.  If $s=1$, we have $F\neq F-\ue_1$
and so we get $|E|\le |F|-1$ as stated.  Assume from now on that
$s\ge 2$ and that the result holds in dimension $s-1$.  For each
$i\in\bZ$, we define
\[
 E_i=\{\ux\in\bZ^{s-1}\,;\, (\ux,i)\in E\}
 \et
 F_i=\{\ux\in\bZ^{s-1}\,;\, (\ux,i)\in F\}.
\]
We also denote by $I$ the set of indices $i\in\bZ$ such that
$F_i\neq\emptyset$, and write $(\ue'_1,\dots,\ue'_{s-1})$ for the
canonical basis of $\bZ^{s-1}$.  Since for each $i\in I$ we have
$E_i \subseteq F_i\cap(F_i-\ue'_1)\cap\cdots\cap(F_i-\ue'_{s-1})$,
the induction hypothesis gives $|E_i| \le |F_i|-
|F_i|^{(s-2)/(s-1)}$. Summing on $i\in I$, upon noting that
$E_i=\emptyset$ when $i\notin I$, this gives $|E| \le |F| - S$ where
$S=\sum_{i\in I} |F_i|^{(s-2)/(s-1)}$. We also have $E_i\subseteq
F_{i+1}$ for each $i\in \bZ$, thus $E_i \subseteq F_i\cap F_{i+1}$
and so
\[
 |E|
 \le \sum_{i\in\bZ} |F_i\cap F_{i+1}|
 = \sum_{i\in\bZ} \big( |F_i|-|F_i\setminus F_{i+1}| \big)
 \le |F|-\Big|\bigcup_{i\in\bZ} (F_i\setminus F_{i+1})\Big|.
\]
Since $\cup_{i\in\bZ} (F_i\setminus F_{i+1}) = \cup_{i\in\bZ} F_i$
contains each $F_i$, we deduce that $|E| \le |F| - M$ where
$M=\max_{i\in I} |F_i|$.  By definition of $S$ and $M$, we have
$SM^{1/(s-1)}\ge \sum_{i\in I} |F_i| = |F|$, and so we get $|E| \le
|F| - \max(S,M) \le |F| - |F|^{(s-1)/s}$ as required.
\end{proof}

\begin{corollary}
With the notation of the proposition, we have $|E| \le (1/s)|F|\log
|F|$.
\end{corollary}

\begin{proof}
Define $g(x)=|F|^x$ for each $x>0$.  The proposition gives $|E|\le
g(1)-g((s-1)/s)$, thus $|E|\le (1/s)g'(\theta) = (1/s) |F|^\theta
\log |F|$ for some real number $\theta$ in the interval
$((s-1)/s,1)$. Since $|F|^\theta \le |F|$, the conclusion follows.
\end{proof}

%
%

\end{document}